\documentclass[11pt]{amsart}
\usepackage{amssymb,amsmath,amsthm,amsfonts,mathrsfs}
\usepackage[hmargin=3cm,vmargin=3.5cm]{geometry}
\usepackage[dvipsnames,table,xcdraw]{xcolor}
\usepackage[colorlinks = true,
            linkcolor = Fuchsia,
            urlcolor  = ForestGreen,
            citecolor = WildStrawberry,
            anchorcolor = blue]{hyperref}

\usepackage{placeins}

\usepackage{stackengine}

\usepackage[all,2cell]{xy} \UseAllTwocells 
\usepackage{tikz-cd}

\usepackage{color}
\usepackage{graphicx,psfrag}
\usepackage{amscd}  
\usepackage{stmaryrd}  %% double square brackets
\usepackage[all,2cell]{xy} \UseAllTwocells \SilentMatrices
\usepackage{tikz}
\usepackage[dvips]{epsfig}
\usepackage{comment}
\usepackage{kbordermatrix}
\usepackage{cancel,scalefnt}
\usepackage{blkarray}
\usepackage{multirow}

\theoremstyle{definition}
\newtheorem{theorem}{Theorem}[section]
\newtheorem{lemma}[theorem]{Lemma}
\newtheorem{remark}[theorem]{Remark}

\newtheorem{proposition}[theorem]{Proposition}

\newtheorem{corollary}[theorem]{Corollary}

\newtheorem{example}[theorem]{Example}

\catcode`~=11 % make LaTeX treat tilde (~) like a normal character
\newcommand{\urltilde}{\kern -.15em\lower .7ex\hbox{~}\kern .04em}  
\catcode`~=13 % revert back to treating tilde (~) as an active character

\usetikzlibrary{arrows,automata}
\usetikzlibrary{decorations.markings}
\usetikzlibrary{decorations.pathmorphing,shapes,decorations.text,shapes.geometric}
\usepackage{multirow}

\newcounter{sarrow}

\usepackage[colorinlistoftodos]{todonotes}

\title[Eventually constant maps for two sets and nilpotent pairs]{Eventually constant maps for two sets and nilpotent pairs}

\author[Weixi Chen]{Weixi Chen}
\address{Department of Mathematics, Johns Hopkins University, Baltimore, MD 21218, USA}
\email{\href{mailto:wchen159@jh.edu}{wchen159@jh.edu}}

\author[Mee Seong Im]{Mee Seong Im}
\address{Department of Mathematics, Johns Hopkins University, Baltimore, MD 21218, USA}
\email{\href{mailto:meeseong@jhu.edu}{meeseong@jhu.edu}}

\author[Catherine Lillja]{Catherine Lillja}
\address{Department of Mathematics, Johns Hopkins University, Baltimore, MD 21218, USA}
\email{\href{mailto:clillja1@jh.edu}{clillja1@jh.edu}}

\author[Nicolas Rugo]{Nicolas Rugo}
\address{Department of Mathematics, Johns Hopkins University, Baltimore, MD 21218, USA}
\email{\href{mailto:nrugo1@jh.edu}{nrugo1@jh.edu}}

\makeatletter
\@namedef{subjclassname@2020}{%
  \textup{2020} Mathematics Subject Classification}

\date{December 4, 2025}

\providecommand{\keywords}[1]{\textbf{\textit{Key words and phrases.}} #1}


\begin{document}

\def\mfb{\mathfrak{b}}

\def\rank{\mathsf{rank}}

\def\Nmn{\mathcal{N}_{m,n}}
\def\NVW{\mathcal{N}(V,W)}

\def\AllPairs{\mathsf{All Pairs}}
\def\E{\mathsf E}
\def\F{\mathbb{F}}
\def\I{\mathsf I}
\def\J{\mathsf J}
\def\M{\mathsf{M}}
\def\R{\mathbb R}
\def\Q{\mathbb Q}
\def\Z{\mathbb Z}
\def\N{\mathbb N}
\def\C{\mathbb C}
\def\S{\mathbb S}
\def\Lin{\mathsf{Lin}}
\def\Nil{\mathsf{Nil}}
\def\SS{\mathbb S}
\def\GL{\mathsf{GL}}
\def\Graph{\mathsf{Graph}}

\def\for{\mathsf{for}}
\def\Hom{\mathsf{Hom}}
\def\End{\mathsf{End}}

\def\Der{\mathsf{Der}}
\def\Pol{\mathsf{Pol}}
\def\Span{\mathsf{Span}}

\newcommand{\dmod}{\mathsf{-mod}}
\newcommand{\comp}{\mathrm{comp}} % components 
\newcommand{\col}{\mathrm{col}}
\newcommand{\adm}{\mathrm{adm}}  % admissible colorings 
\newcommand{\Ob}{\mathrm{Ob}}
\newcommand{\Cob}{\mathsf{Cob}}
\newcommand{\UCob}{\mathsf{UCob}}
\newcommand{\COB}{\mathcal{COB}}
\newcommand{\ECob}{\mathsf{ECob}}
\newcommand{\id}{\mathsf{id}}
\newcommand{\undM}{\underline{M}}
\newcommand{\im}{\mathsf{im}}
\newcommand{\coker}{\mathsf{coker}}
\newcommand{\Aut}{\mathsf{Aut}}
\newcommand{\tripod}{\mathsf{Td}}
\newcommand{\BBC}{\mathbb{B}(\mathcal{C})}
\newcommand{\Pmod}{\mathrm{pmod}}
\newcommand{\gammaoneR}{\gamma_{1,R}}  % decide on notation
\newcommand{\gammaoneRbar}{\overline{\gamma}_{1,R}} % decide on notation
\newcommand{\gammaoneRprime}
{\gamma'_{1,R}}
\newcommand{\gammaoneRbarprime}
{\overline{\gamma}'_{1,R}}
\newcommand{\qbinom}[3]{\genfrac{[}{]}{0pt}{}{#1}{#2}_{#3}}

\def\l{\lbrace}
\def\r{\rbrace}
\def\o{\otimes}
\def\lra{\longrightarrow}
\def\ed{\mathsf{ed}}
\def\Ext{\mathsf{Ext}}
\def\ker{\mathsf{ker}}
\def\mf{\mathfrak} 
\def\mcC{\mathcal{C}}
\def\mcS{\mathcal{S}}  % set of pairs  
\def\mcQC{\mathcal{QC}}
\def\mcA{\mathcal{A}}
\def\mcE{\mathcal{E}}
\def\mcF{\mathcal{F}}
\def\mcN{\mathcal{N}}
\def\Fr{\mathsf{Fr}}  % free module 

\def\bbn{\mathbb{B}^n}
\def\ovb{\overline{b}}
\def\tr{{\sf tr}} 
\def\det{{\sf det }} 
\def\one{\mathbf{1}}   % unit  object of category 
\def\kk{\mathbf{k}}  %% base field  
\def\gdim{\mathsf{gdim}}  %% graded dimension 
\def\rk{\mathsf{rk}}
\def\IET{\mathsf{IET}}
\def\SAF{\mathsf{SAF}}

\newcommand{\indexw}{\R_{>0}} %subscript for weighted foams

\newcommand{\brak}[1]{\ensuremath{\left\langle #1\right\rangle}}
\newcommand{\oplusop}[1]{{\mathop{\oplus}\limits_{#1}}}
\newcommand{\addfigure}{\vspace{0.1in} \begin{center} {\color{red} ADD FIGURE} \end{center} \vspace{0.1in} }
\newcommand{\add}[1]{\vspace{0.1in} \begin{center} {\color{red} ADD FIGURE #1} \end{center} \vspace{0.1in} }
\newcommand{\vspin}{\vspace{0.1in} }

\newcommand\circled[1]{\tikz[baseline=(char.base)]{\node[shape=circle,draw,inner sep=1pt] (char) {${#1}$};}} %shifted dots

% redefine emptyset symbol 
\let\oldemptyset\emptyset
\let\emptyset\varnothing

%%%%% Offset in TOC
\let\oldtocsection=\tocsection
\let\oldtocsubsection=\tocsubsection
\renewcommand{\tocsection}[2]{\hspace{0em}\oldtocsection{#1}{#2}}
\renewcommand{\tocsubsection}[2]{\hspace{1em}\oldtocsubsection{#1}{#2}}

\renewcommand{\kbldelim}{(}% Left delimiter
\renewcommand{\kbrdelim}{)}% Right delimiter

% to insert comments 
\def\MK#1{{\color{red}[MK: #1]}}
\def\bfred#1{{\color{red}#1}}

%\pgfdeclarelayer{background}
%\pgfdeclarelayer{foreground}
%\pgfsetlayers{background, foreground}
%\input{foamstyles.tikzstyles}

\begin{abstract}
We give a bijective correspondence between the number of nilpotent matrices over a Boolean semiring and the number of directed acyclic graphs on ordered vertices. 
We then enumerate pairs of maps between two finite sets whose composites are eventually constant by forming a bijection that relates a pair of such maps with a spanning tree in a complete bipartite graph, and an edge of said tree. This generalizes the main principle of A. Joyal's proof of Cayley's formula.
Finally, we generalize T. Leinster's work by considering a pair of finite-dimensional vector spaces and show a bijectivity between a nilpotent pair of maps and a balanced vector with the hom spaces between them. This leads us to an elegant formula for the number of nilpotent pairs.
\end{abstract}

\maketitle
\tableofcontents

%%%%%%%%%%%%%%%%%%%%
%
% Introduction 
%
%%%%%%%%%%%%%%%%%%%%

\section{Introduction}
\label{sec_intro} 
The geometry, combinatorics and topology of the resolution of the nilpotent cone are central in geometric and combinatorial representation theory and low-dimensional topology. 
They provide profound and important geometric and algebraic structures in the study of nilpotent elements in Lie (super)algebras and representation theory.
The nilpotent elements appear in the moduli space of vector bundles on a Riemann surface (as nilpotent endomorphisms or Higgs fields), as a singular and  reducible space~\cite{HH22,FSS18,BGGH18}.
The nilpotent cone also gives rise to the theory of Grothendieck--Springer and Springer resolutions~\cite{CG97,Spr76_trig,Ste76,Gin97_geom,Dol84,MR3836769,MR3312842,Im_Scrimshaw_parabolic}, as the Springer fibers of the 2-Jordan block have deep connections to 2-row standard Young diagrams, as well as crossingless cups and rays~\cite{SW12,ILW_Proc_AMS,ILWquiver,ILW19}.

Let $X$ be an $n$-dimensional vector space over a field $k$. T.~Leinster in \cite{Lei21} proved the existence of a bijection between $\mathcal{N}(X)\times X$ and $\End_k(X)$, where $\mathcal{N}(X)$ is the set of nilpotent elements on $X$. Working over a finite field, where $|k|= q$, the result specializes to a theorem of Fine--Herstein \cite[Theorem 1]{FH58}, which states that the number of $n\times n$ nilpotent matrices is $q^{n(n-1)}$, which equals $|X|^{n-1}$. This is equivalent to the probability that a random matrix be nilpotent is $q^{-n}$. Leinster's argument requires very little calculation, and exposes a deep, structural pattern.

Fine--Herstein in \cite[Theorem 2]{FH58} enumerate the number of nilpotent matrices over rings of the form $\mathbb{Z}_a$, where $a$ is any positive integer. One can prove the crux of Fine--Herstein, which is finding nilpotent matrices over $\mathbb{F}_p$, using Leinster's elegant, computation-free techniques~\cite{Lei21}.

Our main theorems in this paper are enumerating the eventually constant pairs of maps between two finite sets (Theorem~\ref{thm_num_eventually_const_pairs}), showing a bijective correspondence between a nilpotent pair on two finite-dimensional vector spaces and a balanced vector (which we define in Section~\ref{subsection_nilp_pairs_bal_vec}) and their hom spaces  (Theorem~\ref{thm_gen_nilpotent_two_vs}), and enumerating the number of nilpotent pairs (Theorem~\ref{thm_dim_nilpotent_pairs}).  
The last two theorems in Section~\ref{section_pairs_maps_nilp_finite_field} generalize the result and approach of T.~Leinster to counting nilpotent endomorphisms of a vector space over a finite field.

We now give an overview of the paper. In Section~\ref{section_enumeration_nil_Bool_semiring}, we prove that $n\times n$ nilpotent matrices over a Boolean semiring is enumerated by the number of directed acyclic graphs on $n$ ordered vertices (Proposition~\ref{prop_nilp_bool_semiring}). In Section~\ref{subsection_backgr_graph_thy}, we give a brief background on graph theory, and in Section~\ref{subsect_enum_eventually_const}, we enumerate eventually constant pairs of functions (Theorem~\ref{thm_num_eventually_const_pairs}). A related number appears in the paper~\cite{DK25}, i.e., given a pair $(f,g)$ of eventually constant maps on sets $X$ and $Y$ and removing one of the two oriented edges between fixed points $x_0,y_0$ in $X, Y$, respectively, gives us a spanning tree in the complete bipartite graph $K(m,n)$. The number of spanning trees in $K(m,n)$ is $m^{n-1} n^{m-1}$, but our count is a bit different.

In Section~\ref{subsection_lin_algebra}, we provide the necessary linear algebra for pairs of vector spaces and then we give an explicit bijection between nilpotent pairs and a balanced vector with the hom spaces between the two vector spaces over any field in Section~\ref{subsection_nilp_pairs_bal_vec} (Theorem~\ref{thm_gen_nilpotent_two_vs}). In Section~\ref{subsection_nilp_pairs_bal_fin_field}, we enumerate nilpotent pairs over a finite field (Theorem~\ref{thm_nilpotent_pairs_fin_field}), which simplifies to be an elegant and simple expression (Theorem~\ref{thm_dim_nilpotent_pairs}).
In Section~\ref{subsection_nilp_pairs_bal_vec_char_q}, we enumerate the nilpotent pairs with length $\ell$ balanced vectors (Theorem~\ref{thm_cardinality_nilp_triple}). In Section~\ref{subsection_limiting_case}, we study the limit of the probability of a nilpotent pair when the dimension of one of the vector spaces is fixed and as the dimension of the other tends towards infinity (Theorem~\ref{thm_m_fixed_n_large}), as well as the limit as both dimensions tend towards infinity (Proposition~\ref{prop_m_equals_n_large}). 
In Section~\ref{section_appendix}, we provide two \texttt{Mathematica} codes that produce all nilpotent matrices over the Boolean semiring.

This is a slight extension of \cite{CIKLR25}. Complete proofs as well as detailed examples are provided throughout this paper.

\section*{Acknowledgments}
The authors are grateful to Mikhail Khovanov for helpful suggestions and guidance. The authors would also like to recognize Haihan Wu and Matthew Hamil for constructive conversations. The authors would like to thank the Department of Mathematics and Erin Kathleen Rowe and Jasmine SharDae' Jenkins at the Dean's Office of Krieger School of Arts $\&$ Sciences at Johns Hopkins University for their support. The authors are partially supported by Simons Collaboration Award 994328.

\section{Nilpotent maps over \texorpdfstring{$\mathbb{B}$}{B}}
\label{section_enumeration_nil_Bool_semiring}

Let $\mathbb{B} = \{ 0,1: 1+1 = 1\}$, the Boolean semiring. The number of idempotents in $\M_n(\mathbb{B})$ was proved in \cite{Butler1972}. We give a similar result for $\M_n(\mathbb{B})$ by counting the number of nilpotent matrices. Let $\mcN_n(\mathbb{B})$ be the set of nilpotent matrices over $\mathbb{B}$.

\begin{lemma}
Let $A = (A_{ij})\in \mcN_n(\mathbb{B})$. Then the diagonal coordinates of $A$ all equal 0. If the coordinate $A_{ij}=1$, then $A_{ji}=0$. A Boolean matrix is nilpotent if and only if  there is no sequence 
$i_1$, $i_2$, $\ldots$, $i_k$ satisfying $A_{i_1,i_2}$ $=$ $A_{i_2,i_3}$ $=$ $\ldots$ $=$ $A_{i_k,i_1}=1$.  
\end{lemma}

\begin{proof}
Let $A \in \mathcal{N}(\mathbb{B})$. For a contradiction, suppose $A_{ii}=1$.
If $A^2=0$, then $(A^2)_{ii}$ is of the form $\sum_{j=1}^{n} A_{ij}A_{ji} = A_{ii}^2 + \text{other terms} = 1+ \text{other terms}=1$ since we are working over a Boolean semiring. This is a contradiction, so $A_{ii}$ must be zero. Now if $A^k=0$ for some $k>0$, then 
$(A^k)_{ii} 
= \sum_{j=1}^{n} (A^{k-1})_{ij} A_{ji} 
= (A^{k-1})_{ii} A_{ii} + \mbox{other terms}$. If $A_{ii}\not=0$, then $(A^{k-1})_{ii}\not=0$ by complete induction. This is a contradiction.
So $A_{ii}=0$. 

Now let $A_{ij}=1$, where $i<j$, and suppose $A_{ji}=1$ also. Then $(A^k)_{ii} = \sum_{l=1}^{n} (A^{k-1})_{i l}A_{l i}$. 
If $k-1$ is odd, then this sum is equal to $A_{ij}A_{ji}+ \text{other terms} = 1$. 
If $k-1$ is even, then the $(i,i)$-coordinate of $A^{k-1}$ is $1$. So the sum is equal to $A_{ii}^2 + \text{other terms}=1$. 
Since this is true for an arbitrary $k$, we see that $A$ cannot be a nilpotent matrix. Thus $A_{ji}$ must be $0$.
\end{proof}

\begin{proposition}
\label{prop_nilp_bool_semiring}
The set of $n\times n$ nilpotent matrices $\mathcal{N}_n(\mathbb{B})$ 
    over a Boolean semiring is enumerated by the number of directed acyclic graphs on $n$ ordered vertices.
\end{proposition}

We will write directed acyclic graphs as DAGs.

\begin{proof}
    We show a bijective correspondence between DAGs and nilpotent matrices over $\mathbb{B}$.

    Let $\Gamma$ be a DAG. Label the vertices of $\Gamma$ from $1$ to $n$. Define $A = (A_{ij})$ to be the $n\times n$ matrix such that: 
    \begin{align*}
        A_{ij}= 
            \begin{cases}
                1 &\text{if there is a directed edge from vertex }j\text{ to }i, \\
                0 &\text{otherwise}.
            \end{cases}
    \end{align*}
    This is called the adjacency matrix of $\Gamma$, and is well-known to be a bijection.
    
    Consider $(A^2)_{ij} = \sum_{k=1}^{n}A_{i k}A_{kj}$. Then $(A^2)_{ij}=1$ if and only if there exists at least one path from $j$ to $i$ of length $2$. Working inductively, $(A^n)_{ij}=\sum_{k=1}^n (A^{n-1})_{ik}A_{kj}=1$ if and only if there exists at least one path from $j$ to $i$ of length $n$. Note that a DAG cannot have an infinite length path. Therefore there exists an $N \gg 0 $ such that $(A^N)_{ij}=0$ for each $1\le i,j\le n$. 

    Conversely, suppose there is a cycle in $\Gamma$. Then we can get a path of arbitrary length by tracing the cycle. This will produce a nonzero entry in $A^n$ for any positive integer $n$, so $A$ cannot be nilpotent.
\end{proof}

\begin{remark}
The codes in Figures~\ref{fig_code_boolean} and \ref{fig_nilpotent_mat_only} in Appendix~\ref{section_appendix} provide the enumeration of the cardinality of $\mathcal{N}_n(\mathbb{B})$. 
\end{remark}

\begin{lemma}
\label{lemma_Boolean_nilpotent}
The enumeration in Proposition~\ref{prop_nilp_bool_semiring} follows the number of directed acyclic digraphs (or DAGs) with $n$ labeled vertices.  Since this number follows the recurrence relation: 
\begin{equation}
\label{eqn_recurrence_reln_nilpot_Bool}
a_0=1  
\quad
\mbox{ and } 
\quad 
a_n  =  
\displaystyle{\sum_{k=1}^{n}}   (-1)^{k-1} 
\binom{n}{k} 2^{k(n-k)}
a_{n-k},
\end{equation}
the number of $n\times n$ nilpotent matrices over a Boolean semiring is $a_n$. 
\end{lemma}

\begin{proof}
See \cite{oeisA003024} for more detail.
\end{proof}

\begin{corollary}
The probability that an $n\times n$ operator is nilpotent over $\mathbb{B}$ is $a_n/2^{n^2}$, where $a_n$ is given in~\eqref{eqn_recurrence_reln_nilpot_Bool}.
\end{corollary}

\begin{proof}
Since there are $2^{n^2}$ Boolean-valued $n\times n$ matrices, the result follows by Lemma~\ref{lemma_Boolean_nilpotent}.
\end{proof}

\begin{remark}
Nilpotent endomorphisms of Boolean free (semi)modules (vector spaces) correspond to DAGs.
\end{remark}

\begin{example}
Since the recurrence relation in Lemma~\ref{lemma_Boolean_nilpotent} follows the sequence $a_0 = 1$, $a_1 = 1$, $a_2 = 3$, $a_3 = 25$, $a_4 = 543$, $a_5 = 29,281$, $a_6 = 3,781,503$, 
    there are one $1\times 1$ nilpotent matrix over $\mathbb{B}$ (namely, the zero matrix), three $2\times 2$ nilpotent matrices over $\mathbb{B}$:
\[
\begin{pmatrix}
0 & 0 \\ 
0 & 0 \\ 
\end{pmatrix}, 
\qquad 
\begin{pmatrix}
0 & 0 \\ 
1 & 0 \\ 
\end{pmatrix}, 
\qquad 
\begin{pmatrix}
0 & 1 \\ 
0 & 0 \\ 
\end{pmatrix},
\]
and twenty-five $3\times 3$ nilpotent operators in $\mathcal{N}(\mathbb{B})$:  
\begin{align*}
&\begin{pmatrix}
0 & 1 & 0 \\ 
0 & 0 & 0 \\ 
0 & 0 & 0 \\ 
\end{pmatrix},
        \hspace{2mm}
\begin{pmatrix}
0 & 0 & 0 \\ 
0 & 0 & 1 \\ 
0 & 0 & 0 \\ 
\end{pmatrix},
        \hspace{2mm}
\begin{pmatrix}
0 & 0 & 1 \\ 
0 & 0 & 0 \\ 
0 & 0 & 0 \\ 
\end{pmatrix},
        \hspace{2mm}
\begin{pmatrix}
0 & 1 & 1 \\ 
0 & 0 & 0 \\ 
0 & 0 & 0 \\ 
\end{pmatrix},
        \hspace{2mm}
\begin{pmatrix}
0 & 1 & 0 \\ 
0 & 0 & 1 \\ 
0 & 0 & 0 \\ 
\end{pmatrix},
        \hspace{2mm}
\begin{pmatrix}
0 & 0 & 1 \\ 
0 & 0 & 1 \\ 
0 & 0 & 0 \\ 
\end{pmatrix}, 
    \\ 
&\begin{pmatrix}
0 & 0 & 0 \\ 
1 & 0 & 0 \\ 
0 & 0 & 0 \\ 
\end{pmatrix},
        \hspace{2mm}
\begin{pmatrix}
0 & 0 & 0 \\ 
0 & 0 & 0 \\ 
0 & 1 & 0 \\ 
\end{pmatrix},
        \hspace{2mm}
\begin{pmatrix}
0 & 0 & 0 \\ 
0 & 0 & 0 \\ 
1 & 0 & 0 \\ 
\end{pmatrix},
        \hspace{2mm}
\begin{pmatrix}
0 & 0 & 0 \\ 
1 & 0 & 0 \\ 
1 & 0 & 0 \\ 
\end{pmatrix},
        \hspace{2mm}
\begin{pmatrix}
0 & 0 & 0 \\ 
1 & 0 & 0 \\ 
0 & 1 & 0 \\ 
\end{pmatrix},
        \hspace{2mm}
\begin{pmatrix}
0 & 0 & 0 \\ 
0 & 0 & 0 \\ 
1 & 1 & 0 \\ 
\end{pmatrix}, 
    \\ 
&\begin{pmatrix}
0 & 1 & 0 \\ 
0 & 0 & 0 \\ 
1 & 0 & 0 \\ 
\end{pmatrix},
        \hspace{2mm}
\begin{pmatrix}
0 & 0 & 0 \\ 
0 & 0 & 1 \\ 
1 & 0 & 0 \\ 
\end{pmatrix},
        \hspace{2mm}
\begin{pmatrix}
0 & 0 & 1 \\ 
1 & 0 & 0 \\ 
0 & 0 & 0 \\ 
\end{pmatrix},
        \hspace{2mm}
\begin{pmatrix}
0 & 0 & 1 \\ 
0 & 0 & 0 \\ 
0 & 1 & 0 \\ 
\end{pmatrix},
        \hspace{2mm}
\begin{pmatrix}
0 & 1 & 0 \\ 
0 & 0 & 0 \\ 
0 & 1 & 0 \\ 
\end{pmatrix},
        \hspace{2mm}
\begin{pmatrix}
0 & 0 & 0 \\ 
1 & 0 & 1 \\ 
0 & 0 & 0 \\ 
\end{pmatrix},
    \\
&\begin{pmatrix}
0 & 1 & 1 \\ 
0 & 0 & 1 \\ 
0 & 0 & 0 \\ 
\end{pmatrix}, 
\begin{pmatrix}
0 & 0 & 0 \\ 
1 & 0 & 0 \\ 
1 & 1 & 0 \\ 
\end{pmatrix}, 
\begin{pmatrix}
0 & 1 & 1 \\ 
0 & 0 & 0 \\ 
0 & 1 & 0 \\ 
\end{pmatrix},
\begin{pmatrix}
0 & 0 & 1 \\ 
1 & 0 & 1 \\ 
0 & 0 & 0 \\ 
\end{pmatrix},
\begin{pmatrix}
0 & 0 & 0 \\ 
1 & 0 & 1 \\ 
1 & 0 & 0 \\ 
\end{pmatrix},
\begin{pmatrix}
0 & 1 & 0 \\ 
0 & 0 & 0 \\ 
1 & 1 & 0 \\ 
\end{pmatrix},
\begin{pmatrix}
0 & 0 & 0 \\ 
0 & 0 & 0 \\ 
0 & 0 & 0 \\ 
\end{pmatrix}.
\end{align*}
The directed acyclic graphs corresponding to $n=1$ and $n=2$ are given in Figure~\ref{fig_00001} left and right, respectively.
\begin{figure}
    \centering
\begin{tikzpicture}[scale=0.5,decoration={
    markings,
    mark=at position 0.5 with {\arrow{>}}}]
\begin{scope}[shift={(0,0)}]
%\draw[thin,yellow] (0,0) grid (4,4);

\node at (0,2.65) {$1$};

\draw[thick,fill] (0.2,2) arc (0:360:2mm);

\end{scope}

\begin{scope}[shift={(6.5,0)}]
%\draw[thin,yellow] (0,0) grid (4,4);

\node at (0,2.65) {$1$};

\draw[thick,fill] (0.2,2) arc (0:360:2mm);

\node at (2,2.65) {$2$};

\draw[thick,fill] (2.2,2) arc (0:360:2mm);

\end{scope}

\begin{scope}[shift={(13,0)}]
%\draw[thin,yellow] (0,0) grid (4,4);

\node at (0,2.65) {$1$};

\draw[thick,fill] (0.2,2) arc (0:360:2mm);

\node at (2,2.65) {$2$};

\draw[thick,fill] (2.2,2) arc (0:360:2mm);

\draw[thick,->] (0.5,2) -- (1.5,2);

\end{scope}

\begin{scope}[shift={(19.5,0)}]
%\draw[thin,yellow] (0,0) grid (4,4);

\node at (0,2.65) {$1$};

\draw[thick,fill] (0.2,2) arc (0:360:2mm);

\node at (2,2.65) {$2$};

\draw[thick,fill] (2.2,2) arc (0:360:2mm);

\draw[thick,<-] (0.5,2) -- (1.5,2);

\end{scope}

\end{tikzpicture}
    \caption{There are $1$ directed acyclic graph (DAG) on $1$ (ordered) vertex, and $3$ DAGs on $2$ ordered vertices.}
    \label{fig_00001}
\end{figure}
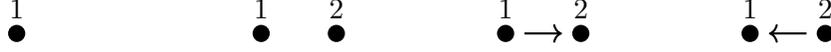
The DAGs associated to the elements in $\mathcal{N}_3(\mathbb{B})$ for $n=3$ are given in Figure~\ref{fig_00002}.
\end{example}

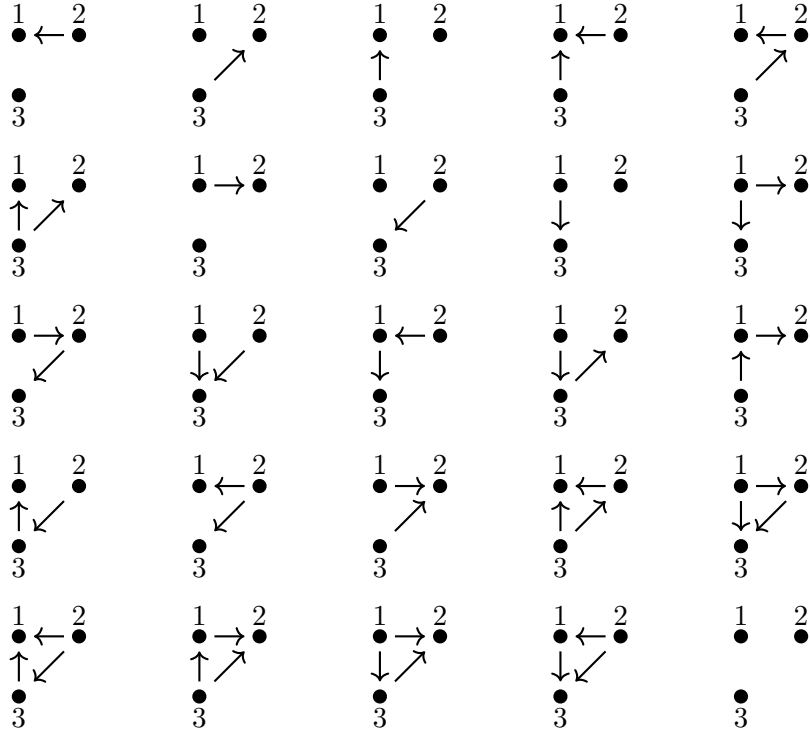
\begin{figure}
    \centering
\begin{tikzpicture}[scale=0.40,decoration={
    markings,
    mark=at position 0.5 with {\arrow{>}}}]

%%%%%%%%%%%%%%%%%%%%%%%%%%%%%%%%%%%%%%%%%%%%%%%%%

\begin{scope}[shift={(0,0)}]
%\draw[thin,yellow] (0,0) grid (4,4);

\node at (0,2.70) {$1$};

\draw[thick,fill] (0.2,2) arc (0:360:2mm);

\node at (2,2.70) {$2$};

\draw[thick,fill] (2.2,2) arc (0:360:2mm);

\draw[thick,fill] (0.2,0.0) arc (0:360:2mm);

\node at (0,-0.70) {$3$};

\draw[thick,<-] (0.5,2) -- (1.5,2);

\end{scope}

\begin{scope}[shift={(6,0)}]
%\draw[thin,yellow] (0,0) grid (4,4);

\node at (0,2.70) {$1$};

\draw[thick,fill] (0.2,2) arc (0:360:2mm);

\node at (2,2.70) {$2$};

\draw[thick,fill] (2.2,2) arc (0:360:2mm);

\draw[thick,fill] (0.2,0.0) arc (0:360:2mm);

\node at (0,-0.70) {$3$};

\draw[thick,->] (0.5,0.5) -- (1.5,1.5);

\end{scope}

\begin{scope}[shift={(12,0)}]
%\draw[thin,yellow] (0,0) grid (4,4);

\node at (0,2.70) {$1$};

\draw[thick,fill] (0.2,2) arc (0:360:2mm);

\node at (2,2.70) {$2$};

\draw[thick,fill] (2.2,2) arc (0:360:2mm);

\draw[thick,fill] (0.2,0.0) arc (0:360:2mm);

\node at (0,-0.70) {$3$};

\draw[thick,->] (0,0.5) -- (0,1.5);

\end{scope}

\begin{scope}[shift={(18,0)}]
%\draw[thin,yellow] (0,0) grid (4,4);

\node at (0,2.70) {$1$};

\draw[thick,fill] (0.2,2) arc (0:360:2mm);

\node at (2,2.70) {$2$};

\draw[thick,fill] (2.2,2) arc (0:360:2mm);

\draw[thick,fill] (0.2,0.0) arc (0:360:2mm);

\node at (0,-0.70) {$3$};

\draw[thick,->] (1.5,2) -- (0.5,2);

\draw[thick,->] (0,0.5) -- (0,1.5);

\end{scope}

\begin{scope}[shift={(24,0)}]
%\draw[thin,yellow] (0,0) grid (4,4);

\node at (0,2.70) {$1$};

\draw[thick,fill] (0.2,2) arc (0:360:2mm);

\node at (2,2.70) {$2$};

\draw[thick,fill] (2.2,2) arc (0:360:2mm);

\draw[thick,fill] (0.2,0.0) arc (0:360:2mm);

\node at (0,-0.70) {$3$};

\draw[thick,->] (1.5,2) -- (0.5,2);

\draw[thick,->] (0.5,0.5) -- (1.5,1.5);

\end{scope}

%%%%%%%%%%%%%%%%%%%%%%%%%%%%%%%%%%%%%%%%%%%%%%%%

\begin{scope}[shift={(0,-5)}]
%\draw[thin,yellow] (0,0) grid (4,4);

\node at (0,2.70) {$1$};

\draw[thick,fill] (0.2,2) arc (0:360:2mm);

\node at (2,2.70) {$2$};

\draw[thick,fill] (2.2,2) arc (0:360:2mm);

\draw[thick,fill] (0.2,0.0) arc (0:360:2mm);

\node at (0,-0.70) {$3$};

\draw[thick,->] (0,0.5) -- (0,1.5);

\draw[thick,->] (0.5,0.5) -- (1.5,1.5);

\end{scope}

\begin{scope}[shift={(6,-5)}]
%\draw[thin,yellow] (0,0) grid (4,4);

\node at (0,2.70) {$1$};

\draw[thick,fill] (0.2,2) arc (0:360:2mm);

\node at (2,2.70) {$2$};

\draw[thick,fill] (2.2,2) arc (0:360:2mm);

\draw[thick,fill] (0.2,0.0) arc (0:360:2mm);

\node at (0,-0.70) {$3$};

\draw[thick,->] (0.5,2) -- (1.5,2);

\end{scope}

\begin{scope}[shift={(12,-5)}]
%\draw[thin,yellow] (0,0) grid (4,4);

\node at (0,2.70) {$1$};

\draw[thick,fill] (0.2,2) arc (0:360:2mm);

\node at (2,2.70) {$2$};

\draw[thick,fill] (2.2,2) arc (0:360:2mm);

\draw[thick,fill] (0.2,0.0) arc (0:360:2mm);

\node at (0,-0.70) {$3$};

\draw[thick,<-] (0.5,0.5) -- (1.5,1.5);

\end{scope}

\begin{scope}[shift={(18,-5)}]
%\draw[thin,yellow] (0,0) grid (4,4);

\node at (0,2.70) {$1$};

\draw[thick,fill] (0.2,2) arc (0:360:2mm);

\node at (2,2.70) {$2$};

\draw[thick,fill] (2.2,2) arc (0:360:2mm);

\draw[thick,fill] (0.2,0.0) arc (0:360:2mm);

\node at (0,-0.70) {$3$};

\draw[thick,<-] (0,0.5) -- (0,1.5);

\end{scope}

\begin{scope}[shift={(24,-5)}]
%\draw[thin,yellow] (0,0) grid (4,4);

\node at (0,2.70) {$1$};

\draw[thick,fill] (0.2,2) arc (0:360:2mm);

\node at (2,2.70) {$2$};

\draw[thick,fill] (2.2,2) arc (0:360:2mm);

\draw[thick,fill] (0.2,0.0) arc (0:360:2mm);

\node at (0,-0.70) {$3$};

\draw[thick,<-] (1.5,2) -- (0.5,2);

\draw[thick,<-] (0,0.5) -- (0,1.5);

\end{scope}

%%%%%%%%%%%%%%%%%%%%%%%%%%%%%%%%%%%%%%%%%%%%%%%%
 
\begin{scope}[shift={(0,-10)}]
%\draw[thin,yellow] (0,0) grid (4,4);

\node at (0,2.70) {$1$};

\draw[thick,fill] (0.2,2) arc (0:360:2mm);

\node at (2,2.70) {$2$};

\draw[thick,fill] (2.2,2) arc (0:360:2mm);

\draw[thick,fill] (0.2,0.0) arc (0:360:2mm);

\node at (0,-0.70) {$3$};

\draw[thick,<-] (1.5,2) -- (0.5,2);

\draw[thick,<-] (0.5,0.5) -- (1.5,1.5);

\end{scope}

\begin{scope}[shift={(6,-10)}]
%\draw[thin,yellow] (0,0) grid (4,4);

\node at (0,2.70) {$1$};

\draw[thick,fill] (0.2,2) arc (0:360:2mm);

\node at (2,2.70) {$2$};

\draw[thick,fill] (2.2,2) arc (0:360:2mm);

\draw[thick,fill] (0.2,0.0) arc (0:360:2mm);

\node at (0,-0.70) {$3$};

\draw[thick,<-] (0,0.5) -- (0,1.5);

\draw[thick,<-] (0.5,0.5) -- (1.5,1.5);

\end{scope}

\begin{scope}[shift={(12,-10)}]
%\draw[thin,yellow] (0,0) grid (4,4);

\node at (0,2.70) {$1$};

\draw[thick,fill] (0.2,2) arc (0:360:2mm);

\node at (2,2.70) {$2$};

\draw[thick,fill] (2.2,2) arc (0:360:2mm);

\draw[thick,fill] (0.2,0.0) arc (0:360:2mm);

\node at (0,-0.70) {$3$};

\draw[thick,<-] (0.5,2) -- (1.5,2);

\draw[thick,<-] (0,0.5) -- (0,1.5);

\end{scope}

\begin{scope}[shift={(18,-10)}]
%\draw[thin,yellow] (0,0) grid (4,4);

\node at (0,2.70) {$1$};

\draw[thick,fill] (0.2,2) arc (0:360:2mm);

\node at (2,2.70) {$2$};

\draw[thick,fill] (2.2,2) arc (0:360:2mm);

\draw[thick,fill] (0.2,0.0) arc (0:360:2mm);

\node at (0,-0.70) {$3$};

\draw[thick,->] (0.5,0.5) -- (1.5,1.5);

\draw[thick,<-] (0,0.5) -- (0,1.5);

\end{scope}

\begin{scope}[shift={(24,-10)}]
%\draw[thin,yellow] (0,0) grid (4,4);

\node at (0,2.70) {$1$};

\draw[thick,fill] (0.2,2) arc (0:360:2mm);

\node at (2,2.70) {$2$};

\draw[thick,fill] (2.2,2) arc (0:360:2mm);

\draw[thick,fill] (0.2,0.0) arc (0:360:2mm);

\node at (0,-0.70) {$3$};

\draw[thick,->] (0,0.5) -- (0,1.5);

\draw[thick,->] (0.5,2) -- (1.5,2);

\end{scope}

%%%%%%%%%%%%%%%%%%%%%%%%%%%%%%%%%%%%%%%%%%%%%%%%

\begin{scope}[shift={(0,-15)}]
%\draw[thin,yellow] (0,0) grid (4,4);

\node at (0,2.70) {$1$};

\draw[thick,fill] (0.2,2) arc (0:360:2mm);

\node at (2,2.70) {$2$};

\draw[thick,fill] (2.2,2) arc (0:360:2mm);

\draw[thick,fill] (0.2,0.0) arc (0:360:2mm);

\node at (0,-0.70) {$3$};

\draw[thick,->] (0,0.5) -- (0,1.5);

\draw[thick,<-] (0.5,0.5) -- (1.5,1.5);

\end{scope}

\begin{scope}[shift={(6,-15)}]
%\draw[thin,yellow] (0,0) grid (4,4);

\node at (0,2.70) {$1$};

\draw[thick,fill] (0.2,2) arc (0:360:2mm);

\node at (2,2.70) {$2$};

\draw[thick,fill] (2.2,2) arc (0:360:2mm);

\draw[thick,fill] (0.2,0.0) arc (0:360:2mm);

\node at (0,-0.70) {$3$};

\draw[thick,<-] (0.5,2) -- (1.5,2);

\draw[thick,<-] (0.5,0.5) -- (1.5,1.5);

\end{scope}

\begin{scope}[shift={(12,-15)}]
%\draw[thin,yellow] (0,0) grid (4,4);

\node at (0,2.70) {$1$};

\draw[thick,fill] (0.2,2) arc (0:360:2mm);

\node at (2,2.70) {$2$};

\draw[thick,fill] (2.2,2) arc (0:360:2mm);

\draw[thick,fill] (0.2,0.0) arc (0:360:2mm);

\node at (0,-0.70) {$3$};

\draw[thick,->] (0.5,2) -- (1.5,2);

\draw[thick,->] (0.5,0.5) -- (1.5,1.5);

\end{scope}

\begin{scope}[shift={(18,-15)}]
%\draw[thin,yellow] (0,0) grid (4,4);

\node at (0,2.70) {$1$};

\draw[thick,fill] (0.2,2) arc (0:360:2mm);

\node at (2,2.70) {$2$};

\draw[thick,fill] (2.2,2) arc (0:360:2mm);

\draw[thick,fill] (0.2,0.0) arc (0:360:2mm);

\node at (0,-0.70) {$3$};

\draw[thick,<-] (0.5,2) -- (1.5,2);

\draw[thick,->] (0,0.5) -- (0,1.5);

\draw[thick,->] (0.5,0.5) -- (1.5,1.5);

\end{scope}

\begin{scope}[shift={(24,-15)}]
%\draw[thin,yellow] (0,0) grid (4,4);

\node at (0,2.70) {$1$};

\draw[thick,fill] (0.2,2) arc (0:360:2mm);

\node at (2,2.70) {$2$};

\draw[thick,fill] (2.2,2) arc (0:360:2mm);

\draw[thick,fill] (0.2,0.0) arc (0:360:2mm);

\node at (0,-0.70) {$3$};

\draw[thick,->] (0.5,2) -- (1.5,2);

\draw[thick,<-] (0,0.5) -- (0,1.5);

\draw[thick,<-] (0.5,0.5) -- (1.5,1.5);

\end{scope}

%%%%%%%%%%%%%%%%%%%%%%%%%%%%%%%%%%%%%%%%%%%%%%%%

\begin{scope}[shift={(0,-20)}]
%\draw[thin,yellow] (0,0) grid (4,4);

\node at (0,2.70) {$1$};

\draw[thick,fill] (0.2,2) arc (0:360:2mm);

\node at (2,2.70) {$2$};

\draw[thick,fill] (2.2,2) arc (0:360:2mm);

\draw[thick,fill] (0.2,0.0) arc (0:360:2mm);

\node at (0,-0.70) {$3$};

\draw[thick,<-] (0.5,2) -- (1.5,2);

\draw[thick,->] (0,0.5) -- (0,1.5);

\draw[thick,<-] (0.5,0.5) -- (1.5,1.5);

\end{scope}

\begin{scope}[shift={(6,-20)}]
%\draw[thin,yellow] (0,0) grid (4,4);

\node at (0,2.70) {$1$};

\draw[thick,fill] (0.2,2) arc (0:360:2mm);

\node at (2,2.70) {$2$};

\draw[thick,fill] (2.2,2) arc (0:360:2mm);

\draw[thick,fill] (0.2,0.0) arc (0:360:2mm);

\node at (0,-0.70) {$3$};

\draw[thick,->] (0.5,2) -- (1.5,2);

\draw[thick,->] (0,0.5) -- (0,1.5);

\draw[thick,->] (0.5,0.5) -- (1.5,1.5);

\end{scope}

\begin{scope}[shift={(12,-20)}]
%\draw[thin,yellow] (0,0) grid (4,4);

\node at (0,2.70) {$1$};

\draw[thick,fill] (0.2,2) arc (0:360:2mm);

\node at (2,2.70) {$2$};

\draw[thick,fill] (2.2,2) arc (0:360:2mm);

\draw[thick,fill] (0.2,0.0) arc (0:360:2mm);

\node at (0,-0.70) {$3$};

\draw[thick,->] (0.5,2) -- (1.5,2);

\draw[thick,<-] (0,0.5) -- (0,1.5);

\draw[thick,->] (0.5,0.5) -- (1.5,1.5);

\end{scope}

\begin{scope}[shift={(18,-20)}]
%\draw[thin,yellow] (0,0) grid (4,4);

\node at (0,2.70) {$1$};

\draw[thick,fill] (0.2,2) arc (0:360:2mm);

\node at (2,2.70) {$2$};

\draw[thick,fill] (2.2,2) arc (0:360:2mm);

\draw[thick,fill] (0.2,0.0) arc (0:360:2mm);

\node at (0,-0.70) {$3$};

\draw[thick,<-] (0.5,2) -- (1.5,2);

\draw[thick,<-] (0,0.5) -- (0,1.5);

\draw[thick,<-] (0.5,0.5) -- (1.5,1.5);
\end{scope}

\begin{scope}[shift={(24,-20)}]
%\draw[thin,yellow] (0,0) grid (4,4);

\node at (0,2.70) {$1$};

\draw[thick,fill] (0.2,2) arc (0:360:2mm);

\node at (2,2.70) {$2$};

\draw[thick,fill] (2.2,2) arc (0:360:2mm);

\draw[thick,fill] (0.2,0.0) arc (0:360:2mm);

\node at (0,-0.70) {$3$};

\end{scope}

%%%%%%%%%%%%%%%%%%%%%%%%%%%%%%%%%%%%%%%%%%%%%%%%

\end{tikzpicture}
    \caption{There are $25$ directed acyclic graphs (DAGs) on $3$ ordered vertices.}
    \label{fig_00002}
\end{figure}

\section{Eventually constant pairs of set-theoretic functions}
\label{section_eventually_const_graphs}

\subsection{Introduction to graph theory}
\label{subsection_backgr_graph_thy}

We need some definitions and lemmas before our theorem. A minimal path is a sequence of vertices in a graph, where each consecutive pair of vertices is connected by an edge. Furthermore, except possibly the endpoints, no vertices are repeated in a minimal path.
A minimal path is called a cycle if there is a repetition. A tree is an unoriented graph with no cycles.

A rooted tree is a tree, 
    together with a choice of vertex, 
    which is called the root. 
    A spanning tree inside an unoriented graph $\Gamma$ 
    is a subgraph $T$ of $\Gamma$ such that $T$ is a tree, and is maximal with respect to being a tree inside $\Gamma$. 
    This statement is equivalent to the statement that $T$ is a tree with as many vertices as $\Gamma$, for trees are connected by definition.

A graph is called bipartite (or bicolorable) if one can color the vertices in two distinct colors, say red and black, such that no two vertices of the same color are adjacent (i.e., connected by an edge). It is well-known in graph theory that $\Gamma$ is bipartite if and only if the cycles of $\Gamma$ are of even length~\cite[Theorem 5.3, page 74]{BM76} or~\cite[pages 48-49]{CCPS11}. The largest bipartite graph with $m$ red vertices and $n$ black vertices is called the complete bipartite graph, which we denote by $K(m, n)$.

\subsection{Eventually constant pairs of functions}
\label{subsect_enum_eventually_const}

Given sets $X,Y$ of cardinalities $m,n$, respectively, we will count the number of pairs $(f,g)$, where $f:X\to Y$, $g:Y\to X$, 
    such that the composition $gf$ is eventually constant, that is, the composite $(gf)^k := gf \circ \cdots \circ gf$ is constant for some 
    $k > 0$. 
    Note that $gf$ is eventually constant if and only if $fg$ is. We call $(f, g)$ an eventually constant pair, and denote the set of such pairs by $P(m, n)$. The number $|P(n,m)|$ of such pairs is divisible by $mn$ 
since the eventually constant condition gives us \textit{final} elements $x_0\in X$ and $y_0\in Y$, where every element eventually maps and $f(x_0)=y_0$, $g(y_0)=x_0$.

This problem reduces to a distributed version of Cayley's formula on counting trees. We look at a variation of Cayley's formula and adopt one to our distributed case.

\begin{lemma}[Cayley's formula]
The cardinality of unrooted trees with vertex set $X$ 
is 
$m^{m-2}$.
\end{lemma}

\begin{proof}
We refer the reader to~\cite[page 16]{A_Joyal81} or ~\cite{Lei21}.
\end{proof}

\begin{lemma}[Number of spanning trees in $K(m, n)$]
There are $m^{n-1}n^{m-1}$ spanning trees in $K(m, n)$.
\end{lemma}

\begin{proof}
See~\cite{FS58} or~\cite{AS90} for a proof.
\end{proof}

Also see Igor Pak's slides~\cite[page 4]{Pak09} on spanning trees and bipartite graphs for more detail.

Next, 
    we state the main theorem of this section, thus generalizing   Cayley's formula.
\begin{theorem}
\label{thm_num_eventually_const_pairs}
Let $X$ and $Y$ be sets of sizes $m$ and $n$, respectively. There are 
\begin{equation}
\label{eqn_eventually_const_pairs}
m^{n-1} n^{m-1} (m + n - 1)
\end{equation}
eventually constant pairs of maps between the finite sets $X$ and $Y$. 
\end{theorem}

\begin{proof}
The basic idea is to first form a directed graph from of the eventually constant pair, $\Gamma:P(m,n) \to \Graph$. The target $\Graph$ is the set of all finite, directed graphs. We will then define $\Phi$ from the image of $\Gamma$ in $\Graph$ to the set of spanning trees in $K(m,n)$, giving a sequence of maps:
% https://q.uiver.app/#q=WzAsMyxbMCwwLCJwKG0sIG4pIl0sWzEsMCwiXFxHYW1tYShwKG0sIG4pKSJdLFszLDAsIlxceyBcXHRleHR7U3Bhbm5pbmcgdHJlZXMgaW4gfSBLKG0sIG4pIFxcfSJdLFswLDEsIlxcR2FtbWEiXSxbMSwyLCJtK24tMToxIl1d
\[\begin{tikzcd}
	{P(m, n)} & {\Gamma(P(m, n))} & {\{ \text{Spanning trees in } K(m, n) \}.}
	\arrow["\Gamma", from=1-1, to=1-2]
	\arrow["\Phi", from=1-2, to=1-3]
\end{tikzcd}\]

If we can establish that $\Gamma$ is one-to-one, and that the preimage of $\Phi$ of a spanning tree in $K(m,n)$ consists of $m+n-1$ graphs, we will have a bijection counting $|P(m, n)|$. We begin by proving the former claim. 

Take an eventually constant pair $(f, g)$ in $P(m, n)$. Then $\Gamma(f, g)$ is defined as a graph with $m$ red vertices labeled by $X$, and $n$ black vertices labeled by $Y$. Draw an oriented edge from every $x \in X$ to its image $f(x) \in Y$, and an edge from each $y \in Y$ to its image $g(y) \in X$. See Figures~\ref{fig_00003} and \ref{fig_00004} for small $m$ and $n$. The graph $\Gamma(f, g)$ is bipartite; there cannot be an edge between two vertices in $X$ (resp. $Y$) because this graph is defined by functions. It is clear that $\Gamma$ is injective, capturing all of the functional information of $f$ and $g$ and thereby establishing a bijection onto its image $P(m, n) \cong \Gamma(P(m, n))$. 

Now, we will look at some finer properties of $\Gamma(f, g)$. A $2k$-cycle in $\Gamma(f, g)$ is equivalent to a $k$-periodic point $x \in X$.
That is, $\Gamma$ has a 2$k$-cycle if and only if there is some point $x \in X$ such that $(gf)^k(x) = x$, and $k$ is the smallest natural number with this property. Therefore, $\Gamma(f, g)$ has exactly one 2-cycle, and no other cycles.

By an abuse of notation, we will write $\Phi(f, g)$ to mean $\Phi\circ \Gamma(f,g)$. 
Now, there are no cycles in $\Phi(f, g)$, and $\Phi(f, g)$ has $m+n$ vertices. So it is a spanning tree, say $T$, inside $K(m, n)$. The only thing to do now is to determine the preimage of $T \subseteq K(m, n)$. Take an edge $e$ in $T$, and remove it. Its endpoints $e_0, e_1$ lie within two disjoint subtrees $T_0, T_1 \subseteq T$, respectively. 

Given two vertices in a tree, there is a unique path between them. Orient the edges of $T_0$ along the unique path to $e_0$, and do the same for $T_1$ and $e_1$. Finally, connect $e_0$ and $e_1$ with a 2-cycle. Since $T$ is bipartite, every graph in $\Phi^{-1}(T)$ is bipartite, and each such graph is the graph of an eventually constant pair. Since there are $m+n-1$ edges in $T$, it is equal to $|\Phi^{-1}(T)|$. Therefore, $|P(m, n)| = |\Gamma(P(m, n))| = (m+n-1) | \{ \text{Spanning trees in } K(m, n) \}| = m^{n-1}n^{m-1}(m+n-1)$. 
\end{proof}

\begin{remark}
Theorem~\ref{thm_num_eventually_const_pairs} establishes a bijection between eventually constant maps and a spanning tree in $K(m,n)$ and an edge of the tree.
\end{remark}

\begin{corollary}
\label{cor_prob_eventually_const}
    The probability that a pair $(f, g)$ of set-theoretic maps, where $f : X \to Y, \; g : Y \to X$, is eventually constant is 
\[ 
    \frac{m+n-1}{mn}.
\]
\end{corollary}

\begin{proof}
Applying Theorem~\ref{thm_num_eventually_const_pairs} and 
since there are 
$m^n n^m$ 
maps between 
$X$ and $Y$, 
we obtain the expression.
\end{proof}

\begin{example}
When $|X|=|Y|=1$, there is $1$ eventually constant map. For $|X|=2$ and $|Y|=1$, there are $2$ eventually constant maps. For $|X|=|Y|=2$, there are $12$ eventually constant maps. See Figures~\ref{fig_00003} and \ref{fig_00004} for more detail. We computed, by hand, $|P(3, 3)| = 405 = 3^4 \times 5$, which also follows from Theorem~\ref{thm_num_eventually_const_pairs}. 
\end{example}

\begin{figure}
    \centering
\begin{tikzpicture}[scale=0.5,decoration={
    markings,
    mark=at position 0.5 with {\arrow{>}}}]

%%%%%%%%%%%%%%%%%%%%%%%%%%%%%%%%%%%%%%%%%%%%%%%%%

\begin{scope}[shift={(0,-1.75)}]
%\draw[thin,yellow] (0,0) grid (4,4);

\node at (-0.75,2.5) {$x_1$};

\node at (-0.85,4) {$X$};

\draw[thick,fill,magenta] (0.2,2.5) arc (0:360:2mm);

\draw[thick,fill] (2.25,2.25) rectangle (2.75,2.75);

\node at (3.4,2.5) {$y_1$};

\node at (3.4,4) {$Y$};

\draw[very thick, dashed, ->, magenta] (0.5,2.75) .. controls (0.75,3.25) and (1.75,3.25) .. (2,2.75);

\draw[thick,<-] (0.5,2.25) .. controls (0.75,1.75) and (1.75,1.75) .. (2,2.25);

\end{scope}

%%%%%%%%%%%%%%%%%%%%%%%%%%%%%%%%%%%%%%%%%

\begin{scope}[shift={(11,0)}]
% \draw[thin,yellow] (0,0) grid (4,4);

\node at (-0.75,2.5) {$x_1$};

\draw[thick,fill,magenta] (0.2,2.5) arc (0:360:2mm);

\node at (-0.75,0) {$x_2$};

\draw[thick,fill,magenta] (0.2,0.0) arc (0:360:2mm);

\draw[thick,fill] (2.25,2.25) rectangle (2.75,2.75);

\node at (3.4,2.5) {$y_1$};

\draw[very thick, dashed, ->,magenta] (0.5,2.75) .. controls (0.75,3.25) and (1.75,3.25) .. (2,2.75);

\draw[thick,<-] (0.5,2.25) .. controls (0.75,1.75) and (1.75,1.75) .. (2,2.25);

\draw[very thick, dashed, ->, magenta] (0.40,0) .. controls (0.65,0) and (2.25,1) .. (2.5,2);

\end{scope}

%%%%%%%%%%%%%%%%%%%%%%%%%%%%%%%%%%%%%%%%%%%%%%%%

\begin{scope}[shift={(18,0)}]
% \draw[thin,yellow] (0,0) grid (4,4);

\node at (-0.75,2.5) {$x_1$};

\draw[thick,fill, magenta] (0.2,2.5) arc (0:360:2mm);

\node at (-0.75,0) {$x_2$};

\draw[thick,fill, magenta] (0.2,0.0) arc (0:360:2mm);

\draw[thick,fill] (2.25,2.25) rectangle (2.75,2.75);

\node at (3.4,2.5) {$y_1$};

\draw[very thick, dashed, ->, magenta] (0.5,2.75) .. controls (0.75,3.25) and (1.75,3.25) .. (2,2.75);

\draw[thick,<-] (0,0.40) .. controls (0,0.75) and (1,2.25) .. (2,2.4);

\draw[very thick, dashed, ->, magenta] (0.40,0) .. controls (0.75,0) and (2.25,1) .. (2.5,2);

\end{scope}

%%%%%%%%%%%%%%%%%%%%%%%%%%%%%%%%%%%%%%%%%%%%%%%%

\end{tikzpicture}
    \caption{Left: there is one eventually constant map when $|X|=|Y|=1$. Right: there are two eventually constant maps when  $|X|=2$ and $|Y|=1$. }
    \label{fig_00003}
\end{figure}
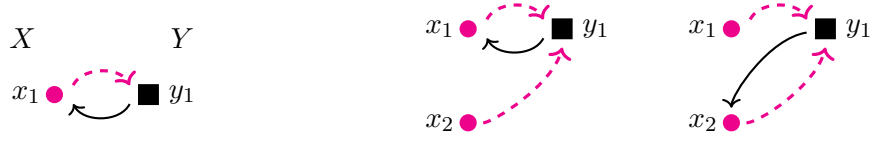

\input{fig_00004}

\section{Nilpotent pairs and balanced vectors}
\label{section_pairs_maps_nilp_finite_field}
Let $k$ be a field.
Let $V$ and $W$ 
    be vector spaces of dimension $m$ and $n$, respectively. 
Consider pairs $(f, g)$ of linear maps, with $f \in \Hom(V,W)$ and $g \in \Hom(W,V)$. 
Let $\mathcal{N}(V)$ be the vector space of nilpotent linear operators on $V$; it is a subspace of $\End(V)$.

In this section, 
we generalize~\cite[Theorem 5]{Lei21}.

\subsection{Linear algebra for two vector spaces}
\label{subsection_lin_algebra}
A complement of a subspace $X$ in a vector space $V$ is a subspace $X'$ satisfying $X \oplus X' = V$.

\begin{lemma}
\label{lemma_canon_complement}
    There is a natural bijection between 
    $\Hom(V, W)$ and complements of 
    $W$ in 
    $V \oplus W$.
\end{lemma}

\begin{proof}
We refer to~\cite[Lemma 1]{Lei21}.
\end{proof}

\begin{lemma}
\label{lem_Isomorphisms}
    Let $f : V \to W$ and $g : W \to V$ be linear maps. Then $gf\in \Aut(V)$ and $fg\in \Aut(W)$ if and only if $f$, $g$ are isomorphisms.
\end{lemma}

\begin{proof}
    The implication from right to left is obvious. So assume $gf \in \Aut(V)$ and $fg \in \Aut(W)$. Then let $f^{-1} = g(fg)^{-1}$ and $g^{-1} = f(gf)^{-1}$. 
\end{proof}

The following lemmas, Lemmas~\ref{lemma_characterization_of_pairs} and~\ref{lemma_bijection_isom}, generalize the study of linear operators to the case of pairs of linear maps. We begin by giving a generalization of the Fitting decomposition.

\begin{lemma}
\label{lemma_characterization_of_pairs}
Let $V$ and $W$ 
be finite-dimensional vector spaces. An ordered pair $(f,g)$ of linear maps, where $f\in \Hom(V,W)$ and $g\in \Hom(W,V)$,
is uniquely characterized by the following data. Write $V_N, V_I$ as subspaces of $V$ such that we have the direct sum decomposition 
$V = V_I \oplus V_N$ and 
    write 
$W_N, W_I$ as subspaces of $W$ such that 
$W = W_I \oplus W_N$,  
    subject to the condition that 
    \begin{equation*}
(gf)|_{V_I}\in \Aut(V_I), \quad
(fg)|_{W_I}\in \Aut(W_I),  \quad
(gf)|_{V_N} \in \mathcal{N}(V_N), \quad
(fg)|_{W_N} \in \mathcal{N}(W_N).
    \end{equation*} 
\end{lemma}

\begin{proof}
Take 
\[ 
V_I = \bigcap_{i \ge 0} \im((gf)^i), 
\quad 
V_N = \bigcup_{i \ge 0} \ker((gf)^i).
\] 
Likewise, take 
\[ 
W_I = \bigcap_{i \ge 0} \im((fg)^i), 
\quad 
W_N = \bigcup_{i \ge 0} \ker((fg)^i). 
\] 
The decomposition $V = V_I \oplus V_N$ and $W = W_I \oplus W_N$ is the statement of Fitting lemma, i.e., \cite[pages 113--114]{Jac89}. Finite-dimensionality is needed since we want $V$ and $W$ to be Noetherian and Artinian. 

In the above subspace decompositions, we claim the following block matrix decompositions for $f$ and $g$:
\[ f= \begin{pmatrix}
    S_1 & 0 \\
    0 & N_1
\end{pmatrix} ,  \quad\quad 
g = 
\begin{pmatrix}
    S_2 & 0 \\
    0 & N_2
\end{pmatrix}, 
\]
where $S_i$ are invertible, and $N_i$ compose to nilpotent operators.

Now, $S_1:=f|_{V_I} :V_I\to W_I$ and $S_2:=g|_{W_I} : W_I \to V_I$ are isomorphisms by Lemma~\ref{lem_Isomorphisms}. As seen in the usual Fitting decomposition, we have $(gf)|_{V_N}$ and $(fg)|_{W_N}$ as nilpotent operators. So $f|_{V_N}$ and $g|_{W_N}$ compose to give nilpotent operators. 

To prove that the bottom-left blocks are zero, we only need to show that $f$ does not map nonzero vectors in $V_I$ into $W_N$. If $v \in V_I$ is such that $fv \in W_N$, then $(fg)^i fv = 0$ for some $i$. That means that $g(fg)^i fv = (gf)^{i+1}v=0$, so $v \in V_N$. Since $V = V_I \oplus V_N$, $v=0$. 

Finally, to show that the top-right blocks are zero, it suffices to show that any $v \in V_N$ such that $fv \in W_I$ is zero. Suppose we had such a $v$. For every $j \ge 0$, there exists a $y_j \in W$ such that $f(v) = (fg)^jy_j$, so $gf(v) = (gf)^j(gy_j) \in \im((gf)^j)$. Since this is true for every $j$, we see that $gf(v) \in V_I$. But $gf$ acts on $V_I$ invertibly. So $v \in V_I$, which implies $v=0$. Therefore, we are done.
\end{proof}

The following is the key lemma for the proof of Theorem~\ref{thm_gen_nilpotent_two_vs}.

\begin{lemma}
\label{lemma_bijection_isom}
There is a (non-canonical) bijection between 
\begin{center}
ordered bases of $\ell$-dimensional vector space $X$ 
    and 
    isomorphisms $X \lra k^{\ell}$.
\end{center}
\end{lemma}

\begin{proof}
Choose a permutation $\sigma \in \Sigma_{\ell}$, where $\Sigma_{\ell}$ is the symmetric group of $\ell$ elements. Denote the standard basis elements of $k^{\ell}$ by  $e_1, \ldots, e_{\ell}$. Suppose we have an ordered basis $(x_1, \ldots , x_{\ell})$ of $X$. Then there is a unique isomorphism $T : X \to k^{\ell}$ such that $T(x_i) = e_{\sigma(i)}$ for each $1 \le i \le \ell$.

On the other hand, given an isomorphism $T : X \to k^{\ell}$, the inverse images of the standard basis $(T^{-1}e_{\sigma(1)}, \ldots , T^{-1}e_{\sigma(\ell)})$ give an ordered basis of $X$.
\end{proof}

\begin{corollary}
\label{corollary_on_order_bases}
One has a bijection between 
\begin{center}
pairs of ordered bases of finite-dimensional vector spaces $V$ and $W$ 
and  \\ 
pairs of isomorphisms 
$(f,g)\in \Hom(V,W)\times \Hom(W,V)$.
\end{center}
\end{corollary}

\subsection{Nilpotent pairs of vector spaces and balanced vectors}
\label{subsection_nilp_pairs_bal_vec}
We will continue to use the notations from Section~\ref{subsection_lin_algebra}.
Consider the endomorphisms 
    $T=gf:V\lra V$ 
and 
    $T'=fg:W\lra W$. 
The pair $(f,g)$ of linear maps is a \emph{nilpotent pair} if $T\in \mathcal{N}(V)$, equivalently, $T'\in \mathcal{N}(W)$. Denote the set of nilpotent pairs $(f,g)$ for spaces $V$ and $W$ 
by 
$\mathcal{N}(V,W)$. 

Let 
$v\in V$.
We form the subspace $T[v]$ of $V$, spanned by the linearly independent vectors 
\[ 
v,Tv,\dots, T^{a-1}(v)\not= 0,  
\] 
with 
    $T^{a}(v)=0$,
and the subspace 
    $T'[fv]$ of $W$, 
    which is the span of linearly independent vectors 
\[ 
fv, T' (fv),
\ldots, (T')^{\ell-1} (fv) \not= 0, 
\] 
with 
$(T')^{\ell}(fv) = 0$.

We say that the vector $v$ is \emph{balanced} if $a = \ell$, i.e., the spaces $T[v]$ and $T'[fv]$ have the same dimension, so that $f$ restricts to an isomorphism $T[v]\stackrel{\simeq }{\lra} T'[fv]$. We call $a = \ell$ the \emph{length} of 
$v$. 
In the case that  
    $\ell=a-1$, so the second subspace has one dimension less than the first, i.e., $\dim T'[fv]=\dim T[v]-1$, we say that $v$ is \emph{unbalanced} in this latter case.

Figure~\ref{fig_00009} 
    shows the structure of images of 
    $v$ 
    under compositions of the maps 
    $f$ and $g$ 
    when $v$ 
    is balanced.

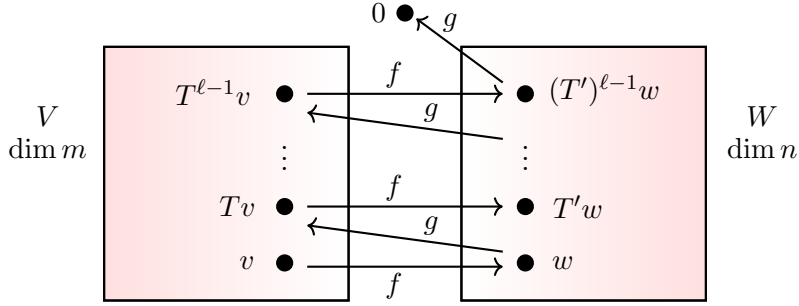
\begin{figure}
    \centering
\begin{tikzpicture}[scale=0.5,decoration={
    markings,
    mark=at position 0.5 with {\arrow{>}}}]
 
%%%%%%%%%%%%%%%%%%%%%%%%%%%%%%%%%%%%%%%%%%%%%%%%%%%%%

\begin{scope}[shift={(0,0)}]
% \draw[thin,yellow] (0,0) grid (5,6);

\draw[thick,shade,left color=pink!50, right color=white] (-1.5,0) rectangle (5,6.75);

\begin{scope}[shift={(8,0)}]
\draw[thick,shade, left color=white, right color=pink!50] (0,0) rectangle (6.5,6.75);
\end{scope}

\draw[thick] (-1.5,0) rectangle (5,6.75);

\node at (-3,4.9) {$V$};

\node at (-3,4.05) {$\dim m$};

\draw[thick,fill] (3.5,5.5) arc (0:360:2mm);

\node at (1.4,5.55) {$T^{\ell - 1}v$};

\node at (3.3,4) {$\vdots$};

\draw[thick,fill] (3.5,2.5) arc (0:360:2mm);

\node at (2.05,2.5) {$Tv$};

\draw[thick,fill] (3.5,1) arc (0:360:2mm);

\node at (2.3,1) {$v$};

\node at (5.8,7.65) {$0$};

\draw[thick,fill] (6.7,7.65) arc (0:360:2mm);

\node at (7.7,7.4) {$g$};

\draw[thick,<-] (6.8,7.5) -- (9.1,5.8);

\node at (6.2,6) {$f$};

\draw[thick,->] (3.9,5.5) -- (9.1,5.5);

\node at (7.2,5.00) {$g$};

\draw[thick,<-] (3.9,5) -- (9.1,4.3);

\node at (6.2,3) {$f$};

\draw[thick,->] (3.9,2.5) -- (9.1,2.5);

\draw[thick,->] (3.9,0.9) -- (9.1,0.9);

\node at (6.2,0.4) {$f$};

\node at (7.2,2.00) {$g$};

\draw[thick,<-] (3.9,2) -- (9.1,1.3);

\end{scope}

%%%%%%%%%%%%%%%%%%%%%%%%%%%%%%%%%%%%%%%%%%%%%%%%%%%%%

\begin{scope}[shift={(8,0)}]
% \draw[thin,yellow] (0,0) grid (5,6);

\draw[thick] (0,0) rectangle (6.5,6.75);

\node at (8,4.9) {$W$};

\node at (8,4.05) {$\dim n$};

\draw[thick,fill] (1.9,5.5) arc (0:360:2mm);

\node at (3.8,5.6) {$(T')^{\ell - 1}w$};

\node at (1.7,4) {$\vdots$};

\draw[thick,fill] (1.9,2.5) arc (0:360:2mm);

\node at (3.1,2.5) {$T'w$};

\draw[thick,fill] (1.9,1) arc (0:360:2mm);

\node at (2.7,1) {$w$};

\end{scope}

%%%%%%%%%%%%%%%%%%%%%%%%%%%%%%%%%%%%%%%%%%%%%%%%%%%%%

\end{tikzpicture}
    \caption{This is an example of a balanced vector $v$ of length $\ell$ and its images under the compositions of the linear maps 
    $f:V\rightarrow W$ 
        and 
    $g:W\rightarrow V$.}
    \label{fig_00009}
\end{figure}

Generalizing \cite[Theorem 5]{Lei21}, we have the following result:

\begin{theorem}
\label{thm_gen_nilpotent_two_vs}
There is a bijection betwen 
\[
\{ (f, g, v): gf\in \mathcal{N}(V), v \mbox{ is balanced} \} 
\mbox{ and }
\Hom(V,W)\times \Hom(W,V).
\]
\end{theorem}

We also refer to \cite[Theorem 3.6]{CIKLR25}.

\begin{proof}
    Let $f, g$ be a nilpotent pair with a balanced vector $v$ of length $\ell$. We have subspaces $T[v]$ and $T'[fv]$ of the same dimension, each equipped with ordered bases $(v, Tv, \ldots, T^{\ell - 1}v)$ and $(fv, T'fv, \ldots, (T')^{\ell - 1} fv)$, respectively. In these bases, $f|_{T[v]}$ is the identity matrix and $g|_{T'[fv]}$ is a nilpotent Jordan block of size $\ell$, i.e.,
    $(J_{\ell})_{ij} = \delta_{i, j+1}, \text{ indices taken modulo }\ell$, by construction.

    We have the complements 
    $T[v]^{\perp}$ and $T'[fv]^{\perp}$ of 
    $T[v]$ and $T'[fv]$ 
    in $V$ and $W$, 
    respectively, by Corollary~\ref{corollary_on_order_bases}.  
This is due to an analogy with Tom Leinster's way of choosing complements. We do not initially get a canonical choice of complement for our subspaces, but we do get restrictions of the given linear maps to some complement. The composites give a canonical choice of the complement of each subspace.    
    With respect to these subspaces, 
    the maps $f$ and $g$ are given by:
    \begin{equation}
    \label{eqn_f_g_nilp_pair}
    f 
    = 
    \kbordermatrix{
                   &  T[v] & T[v]^{\perp}  \\ 
    T'[fv]         &  \I_{\ell}  & A \\
    T'[fv]^{\perp} &  0    & f|_{T[v]^{\perp}}
    }, 
    \hspace{6mm}
    g 
    = 
    \kbordermatrix{
                & T'[fv] & T'[fv]^{\perp} \\ 
    T[v]        & \J_{\ell}  & B \\
    T[v]^{\perp}&     0  & g|_{T'[fv]^{\perp}}
    },
    \end{equation}
    where $\I_{\ell}$ is the $\ell\times \ell$ identity matrix, $\J_{\ell}$ is the $\ell\times \ell$ Jordan block with $1$'s immediately below the diagonal, 
    $A$ is an $\ell\times (m-\ell)$ block, and 
    $B$ is an $\ell\times (n-\ell)$ block.
    Since 
    \[ 
    gf = 
    \begin{pmatrix}
    C & D \\ 
    0 & E \\ 
    \end{pmatrix}, 
    \] 
    where 
    $C$ is a $\dim T[v]\times \dim T[v]$-block,
    $D$ is a $\dim T[v] \times \dim T[v]^{\perp}$-block, and 
    $E$  is a $\dim T[v]^{\perp} \times \dim T[v]^{\perp}$-block, 
        and
    $gf\in \mathcal{N}(V)$ if and only if $C\in\mathcal{N}(T[v])$ and $E\in\mathcal{N}(T[v]^{\perp})$, 
    we see that the restricted 
    $(f|_{T[v]^{\perp}}, g|_{T'[fv]^{\perp}})$ 
    is a nilpotent pair. By Corollary~\ref{corollary_on_order_bases}, 
    the data of subspaces with ordered bases is equivalent to a pair of isomorphisms 
    $S_1 : T[v] \to T'[fv]$ and 
    $S_2 : T'[fv] \to T[v]$. 
    We replace the upper-left blocks of~\eqref{eqn_f_g_nilp_pair} 
    with more generic 
    $S_1$ and $S_2$ 
    blocks
    since the original blocks gave the data of ordered bases of   
    $T[v]$ and   
    $T'[fv]$,   
    and we are simply replacing them with their corresponding isomorphisms.   

    Finally, there may be many choices of complement for $T[v]$ and $T'[fv]$. In order to resolve this, we can apply \ref{lemma_canon_complement}. We observe that $BA \in \Hom(T[v]^\perp, T[v])$ corresponds exactly to a choice of complement for $T[v]$ in $V$. Likewise, $AB \in \Hom(T'[fv]^\perp, T'[fv])$ corresponds exactly to a choice of complement for $T'[fv]$ in $W$. So, up to relabeling the rest of the blocks (for notational simplicity), we have
    \[ 
    f' = \begin{pmatrix}
        S_1 & 0 \\
        0 & N_1
    \end{pmatrix} \quad \mbox{ and } \quad  
    g' = \begin{pmatrix}
        S_2 & 0 \\
        0 & N_2
    \end{pmatrix}, 
    \]
    where the blocks decompose $V, W$ into complementary subspaces. The blocks $N_i$ compose to nilpotent operators, and $S_i$ compose to isomorphisms. 
    By Lemma~\ref{lemma_characterization_of_pairs}, 
    we  obtain a unique pair of linear maps
    \[  
    (f', g') \in \Hom(V, W) \times \Hom(W, V). 
    \]
This concludes the proof.
\end{proof}

\subsection{Nilpotent pairs of vector spaces over a finite field}
\label{subsection_nilp_pairs_bal_fin_field}
Let $V$ and $W$ 
    be $m$ and $n$-dimensional vector spaces, 
    respectively, 
    over the finite field 
$k = \mathbb{F}_q$, 
where $q$ is a prime power.
We want the probability that the pair $(f,g)$ is nilpotent, i.e., $fg\in \mathcal{N}(W)$ for a given 
$m$ and $n$. 
Recall the 
$q$-binomial coefficient
\begin{equation}
\label{eqn_q_binom_coefficient}
\qbinom{m}{r}{q} 
    := 
    \prod_{i=0}^{r-1}(q^m - q^i) \Big/ 
    \prod_{i=0}^{r-1}(q^r - q^i), 
\end{equation}
where if $r=0$, 
    then we define 
$\prod_{i=0}^{r-1}(q^k - q^i):= 1$.
Equation~\eqref{eqn_q_binom_coefficient} 
    represents the number of 
    $r$-dimensional subspaces of an 
    $m$-dimensional vector space over 
    $\mathbb{F}_q$.
\begin{lemma}
\label{lemma_number_rank_r_linear_maps}
    Let $0\le r\le \min\{m,n\}$. 
    The number of rank $r$ linear maps $f:V\rightarrow W$ is given by  
    \[ 
    \qbinom{m}{r}{q} 
    \qbinom{n}{r}{q} 
    \prod_{i=0}^{r-1} (q^r - q^i).
    \]
\end{lemma}
\begin{proof}
    We proceed in two steps.
    \begin{enumerate}
        \item Choose an $r$-dimensional subspace $I$ of $W$ to be the image of $f$. There are exactly $\qbinom{n}{r}{q}$ choices.
        \item Count the number of linear maps from $V$ to $W$ whose image is equal to $I$. This is equal to the number of surjective linear maps from $V$ to $I$. Now choose a basis for $V$ and $W$. Such a linear map can be represented as an $r\times m$ matrix of rank $r$, which means it has full rank. The number of such matrices (by choosing row vectors) is the number of $r$ linearly independent vectors in $V$, which is $\prod_{i=0}^{r-1}(q^m - q^i)$. 
    \end{enumerate}
    Using \eqref{eqn_q_binom_coefficient}, we obtain the desired result.
\end{proof}

Let 
$\mathcal{N}(V,W)$ be the set of nilpotent pairs $(f,g)$ 
as above 
$(f:V\lra W$ and $g:W\lra V)$. 
Let 
$\mathcal{N}_{m,n} 
= |\mathcal{N}(V,W)|$ 
be the cardinality of 
$\mathcal{N}(V,W)$. 

\begin{theorem}
\label{thm_nilpotent_pairs_fin_field}
    The number of nilpotent pairs is given by
    \begin{equation}
    \label{eqn_nilpotent_pairs_Fq}
		\mathcal{N}_{m,n} 
        = 
        \sum_{r=0}^{\min\{m,n\}}  \qbinom{m}{r}{q} \qbinom{n}{r}{q} q^{mn-r} \prod_{i=0}^{r-1}(q^r - q^i),
    \end{equation}
where if $r=0$, then $\prod_{i=0}^{r-1}(q^r - q^i) := 1$. 
\end{theorem}

\begin{proof}
    First, given a linear map $f:V\rightarrow W$ of rank $r$, we count the number of maps $g$ such that the composition $fg$ is nilpotent.
    Without loss of generality, let $\{v_{r+1},\ldots,v_m\}$ be a basis for $\ker f\subset V$. We extend it to a basis $\{ v_1,\ldots, v_m\}$ of $V$. 
    Since $f$ is a linear map, 
    $\{f(v_1),\ldots,f(v_r)\}$ forms a basis of $\im f\subset W$. We now extend it to a basis 
    \[ 
    \{w_1=f(v_1),\ldots, w_r=f(v_r),w_{r+1},\ldots, w_{n}\}
    \] 
    of $W$.
        
    With respect to these two bases for $V$ and $W$, the $n\times m$ matrix representation of $f$ 
    is given by 
        \[ 
        F 
        = 
        \begin{pmatrix}
        \I_r & 0 \\ 
        0    & 0
        \end{pmatrix},
        \]
        where $\I_r$ is the $r\times r$ identity matrix.
        Now let 
        \[
        G = 
        \begin{pmatrix}
        A & B \\
        C & D
        \end{pmatrix}
        \] 
        be the matrix representation of $g$, where $A$ is $r \times r$, $B$ is $r \times (n-r)$, $C$ is $(m-r) \times r$, and $D$ is $(m-r) \times (n-r)$ matrices.
        The matrix representation for $fg$ is the product:
        \[
        FG = 
        \begin{pmatrix} 
        \I_r & 0 \\ 
        0 & 0 
        \end{pmatrix} 
        \begin{pmatrix} 
        A & B \\ 
        C & D 
        \end{pmatrix} 
        = 
        \begin{pmatrix} 
        A & B \\ 
        0 & 0 
        \end{pmatrix},
        \]
        which is a block upper-triangular matrix. Such a matrix is nilpotent if and only if its diagonal blocks are nilpotent, which in our case, $A$ must be nilpotent. By ~\cite[Theorem 5]{Lei21}, the cardinality of such $r\times r$ nilpotent 
        $A\in \M_r(\mathbb{F}_q)$ is $q^{r(r-1)}$. 
        There are no restrictions for blocks $B,C$ and $D$. So the total number of such matrices $G$ (and thus maps $g$) for a fixed $f$ of rank $r$ is given by:
        \[ 
    q^{r(r-1)}  q^{r(n-r)}  q^{r(m-r)}  q^{(m-r)(n-r)} 
        = q^{mn - r}. 
        \]
    Therefore, for any $f$ of rank $r$, 
    there are the same number of linear maps 
    $g$ such that the composition $fg$ is nilpotent.

By Lemma~\ref{lemma_number_rank_r_linear_maps}, we know the number of rank $r$ linear maps between $V$ and $W$. Summing over the rank of $f$, the result follows.
\end{proof}

\begin{proposition}
Equation~\eqref{eqn_nilpotent_pairs_Fq} 
equals  
\[
    \sum_{T\in \Hom(V,W)} 
        q^{mn-\rank(T)}.
\]
\end{proposition}

\begin{proof}
The number of rank $r$ maps $T:V\lra W$ is given by the expression in Lemma~\ref{lemma_number_rank_r_linear_maps}. 
In \eqref{eqn_nilpotent_pairs_Fq}, 
    each map is weighted with the coefficient $q^{mn - r}$.
\end{proof}

\begin{lemma}
\label{lem_basic_facts}
    Let $(f,g)$ 
    be a nilpotent pair. 
    \begin{enumerate}
    \item If $v\in V$ or $w\in W$ is balanced, then every vector in the subspace $T[v]\subset V$ or in $T'[w]\subset W$, respectively, is balanced. 
    \item  If $v\in V$ is unbalanced, 
    then every vector in  
    $T[v]\setminus\{0\}$ 
    is unbalanced. 
    Similarly, 
    if $w\in W$ 
    is unbalanced, 
    then every vector in $T'[w]\setminus\{0\}$ 
    is unbalanced. 
    \item If 
    $v\in V$ 
    is balanced and    
    $f(v) \not= 0$, 
    then the vector 
    $f(v)$ 
    is unbalanced. 
    If 
    $w\in W$ 
    is balanced and 
    $g(w) \not= 0$, 
    then the vector   
    $g(w)$    
    is unbalanced.   
    \item If  
    $v\in V$  
    is unbalanced, 
    then the vector 
    $f(v)\in W$ 
    is balanced. 
    Similarly, 
    if 
    $w\in W$ 
    is  unbalanced, 
        then the vector $g(w)\in V$ is balanced. 
    \end{enumerate}
\end{lemma}

\begin{proof} 
We will now prove the lemma.
    \begin{enumerate}
        \item Suppose $T[v]$ is spanned by $\{v,Tv,\ldots,T^{k}v\}$. Correspondingly, suppose $T'[fv]$ is spanned by $\{fv,T'fv,\ldots,T'^{k}fv\}$. We just need to show that $T^iv$ is balanced for each $i=1,\ldots,k$, and extend linearly. We have $\dim T[T^iv]=k-i+1=\dim T'[fT^iv]$ $=$ $\dim T'[(T')^{i}fv]$. And the result for $w\in W$ follows by symmetry.
        \item Assume $T[v]$ is spanned by $\{v,$ $Tv,$ $\ldots,$ $T^{k}v\}$ 
        but $T'[fv]$ is spanned by $\{fv,$ $\ldots,$ $(T')^{k-1}fv\}$. We need to show $T^iv$ is unbalanced for each $i$, and then extend linearly. 
        This follows from $\dim T[T^iv]=k-i+1$ and $\dim T'[fT^iv] = \dim T'[(T')^ifv]$ $=$ $k-i$. 
        Note that 0 is trivially balanced. The result for $w\in W$ follows by symmetry.
        \item We just need to note that $v$ balanced imply $\dim T[v]$ $=$ $\dim T'[fv]$, 
        and thus 
        $\dim T'[fv]$ $=$ 
        $1+\dim T[Tv]=1+\dim T[gfv]$. 
        The result for $w\in W$ follows by symmetry.
        \item We see that $v$ unbalanced imply $\dim T[v]=1+\dim T'[fv]$, and thus $\dim T'[fv]=\dim T[Tv]=\dim T[gfv]$. The result for $w\in W$ follows by symmetry.
    \end{enumerate}
This concludes the proof.
\end{proof}

We omit the proofs for the following corollaries.

\begin{corollary} 
Let a nilpotent pair 
$(f,g)$ be given.
\begin{enumerate}
\item 
    The set of balanced vectors for the pair 
    $(f,g)$ in 
    $V$ 
    is a union of $T$-stable linear subspaces of 
    $V$. 
    It is likewise for balanced vectors in 
    $W$. 
    \item 
    The set of unbalanced vectors for the pair 
    $(f,g)$ 
    in 
    $V$ 
    union with the zero vector is a union of 
    $T$-stable linear subspaces of 
    $V$. 
    It is likewise for unbalanced vectors in 
    $W$. 
    \end{enumerate}
\end{corollary}

\begin{corollary} Suppose a nilpotent pair $(f,g)$ is given.
Let $v\in V$ be balanced and $v'\in V$ be unbalanced. 
Then the intersection 
$T[v]\cap T[v']$ is the $0$-dimensional vector space. 
\end{corollary}

\begin{corollary} Let a nilpotent pair $(f,g)$ be given.
Let $v\in V$ and $w\in W$ be balanced vectors. 
Then 
\[
T[v] \cap T[g(w)] 
    = 
    \{0 \} \subset V 
 \hspace{2mm}
 \mbox{ and }
 \hspace{2mm}
T'[f(v)] \cap T'[w] 
    = 
    \{ 0 \} \subset W,
\]
and the sums 
    $T[v] + T[g(w)]$ 
        in $V$ 
    and 
    $T'[f(v)]+ T'[w]$ 
        in $W$ 
are isomorphic to their respective direct sums. 
\end{corollary}

\begin{theorem}
\label{thm_dim_nilpotent_pairs}
The number of nilpotent pairs   
    $(f,g)$ 
    is given by 
\begin{equation} 
\label{eqn_num_nil_pairs_elegant_expression}
\mathcal{N}_{m,n}=q^{2mn-m-n}(q^m + q^n - 1). 
\end{equation}
\end{theorem}
See \cite[Theorem 3.17]{CIKLR25} for the proof. 
This expression is much simpler than the right hand side of \eqref{eqn_nilpotent_pairs_Fq}. 
It would be interesting to understand the relationship between 
the count in \eqref{eqn_num_nil_pairs_elegant_expression} 
and that in~\eqref{eqn_eventually_const_pairs}.

Theorem~\ref{thm_dim_nilpotent_pairs} 
gives us a bijection 
\[
\mathcal{N}(V,W)\cdot (v\in V)\cdot (w\in W) = \AllPairs(f,g) \cdot  (u \in V\cup W),
\]
where 
    $\AllPairs(f,g)$ 
denotes all pairs of linear maps,
which would descend to the formula 
\[
\mathcal{N}_{m,n}\, q^m \, q^n  =  q^{2mn} (q^m + q^n - 1).
\]
Equivalently, the theorem gives us a bijection
\[
\Hom(V,W) \times \Hom(W,V) 
    \times (V\cup_0 W) 
    \cong 
    \mathcal{N}(V,W) 
    \times V \times W,
\]
where $V\cup_0 W$ denotes the union of the vector spaces $V$ and $W$ along the $0$ vector.

\vspace{0.07in} 
The number $\Nmn$ of nilpotent pairs is given by \eqref{eqn_num_nil_pairs_elegant_expression}. 
We can multiply by $q^{m+n}$ 
and count nilpotent pairs 
$\NVW$ 
together with a vector $v$ in $V$ and $w$ in $W$. Then we see that the quadruples $(f,g,v,w)$ with nilpotent $(f,g)$ are in a ``bijection" with the data $(f',g')$, where $f',g'$ are any pair of maps plus a vector in $V$ or a vector in $W$ minus the 0 vector. Equivalently, for the latter we can take a vector in the union $V\cup W$ viewed as a subset of $V\oplus W$.

\vspace{0.07in} 

Let us separate $v$ into balanced $v_b$ and unbalanced $v_u$ and likewise for $w's$. 
We have, for nilpotent pairs  
\[
|(f,g,v)|=|(f,g,v_b)| + |(f,g,v_u)|=q^{2mn}+|(f,g,v_u)|,
\]
and likewise for $w$'s. 
We know that the count of nilpotent $(f,g,v_b,w)$ is $q^{2mn}q^n$ and the count of nilpotent 
$(f,g,v,w_b)$ is $q^{2mn}q^m$. 
Adding these two, we get 
\begin{equation}
\label{eq_bb}
|(f,g,v_b,w)| + |(f,g,v,w_b)| 
= q^{2mn}(q^m + q^n).
\end{equation}
The left hand side above equals 
\[
2 |(f,g,v_b,w_b)| + |(f,g,v_b,w_u)|+|(f,g,v_u,w_b)|.
\]
For nilpotent $(f,g)$, since we have the equality, 
\[
|(f,g,v,w)| = q^{2mn}(q^m + q^n - 1), 
\]
subtracting \eqref{eq_bb}, we have
\begin{equation}\label{eq_difference}
|(f,g,v_b,w_b)| = |(f,g,v_u,w_u)| + q^{2mn}.
\end{equation}
That is, the number of nilpotent pairs together with choices of balanced vectors $(v_b,w_b)$ equals the number of nilpotent pairs together with choices of unbalanced vectors $(v_u,w_u)$ plus the number of all pairs of maps $(f',g')$. 

We have now set up a bijection between sets of quadruples on the left hand side of \eqref{eq_difference} and the union of quadruples on the right hand side of \eqref{eq_difference} together with all pairs of maps $(f',g')$.

\begin{remark}
Let us try to understand balanced versus unbalanced vectors for a given nilpotent pair $(f,g)$. Let 
\[
d_1=\dim(\ker(f)), \quad
d_2 = \dim(\ker(fgf)), \hspace{2mm}
\ldots, \quad
d_k =\dim(\ker(f(gf)^{k-1}). 
\]
Then $d_1\le d_2\le \ldots \le d_k$ and $d_k = m = \dim V$ for large $k$. 
Likewise, let 
\[
d_0'=\dim(\ker((gf)^0))=0, \quad
d_1'=\dim(\ker(gf)), \hspace{2mm}
\ldots, \quad
d_k'=\dim(\ker((gf)^k). 
\]
Then $d_1'\le d_2'\le \ldots \le d_k'$ and $d_k' = m$ 
for large $k$. 
These sequences are nested, so that 
\[
0=d_0'\le d_1 \le d_1'\le d_2\le d_2'\leq \ldots \le d_k
\]
Let us use these sequences to count the number of balanced and unbalanced vectors in $V$.
The subspace $\dim(\ker((gf)^0))$ is $0$ and it contains the unique balanced vector $0$. Since $d_1=\dim (\ker f)$, the latter subspace contains $q^{d_1}-1$ unbalanced vectors and one balanced vector $0$. The subspace $\ker (gf)$ has dimension $d_1'$ and contains $q^{d_1'}-q^{d_1}+1 $ balanced vectors.
We get that the number of balanced vectors in $V$ is 
\[
|v_b| 
= 1 - q^{d_1}+q^{d_1'}-q^{d_2}+q^{d_2'}-\ldots 
= \sum_{i=0}^k q^{d_i'} - \sum_{i=1}^k q^{d_i}, 
\]
where $k$ is any number where the sequences $(d_i)$ and $(d_i')$ stabilize. 
% {\bf Say it more carefully. Then write down the number of unbalanced vectors. Also, we don't need this remark for the proof.}
\end{remark}

\begin{figure}
    \centering
\begin{tikzpicture}[scale=0.5,decoration={
    markings,
    mark=at position 0.5 with {\arrow{>}}}]

%%%%%%%%%%%%%%%%%%%%%%%%%%%%%%%%%%%%%%%%%%%%%%%%%%%%%

\begin{scope}[shift={(0,0)}]
%\draw[thin,yellow] (0,0) grid (5,6);

\draw[thick,shade, left color=WildStrawberry!35, right color=white] (-1.5,0) rectangle (5,6.5);

\begin{scope}[shift={(8,0)}]
\draw[thick,shade, left color=white, right color=WildStrawberry!35] (0,0) rectangle (6.5,6.5);
\end{scope}

\begin{scope}[shift={(0,-8)}]
\draw[thick,shade, left color=WildStrawberry!35, right color=white] (0.5,0) -- (-1.5,6.5) -- (5,6.5) -- (7,0) -- (0.5,0);
\end{scope}

\begin{scope}[shift={(10,-8)}]
\draw[thick,shade, left color=white, right color=WildStrawberry!35] (0,0) -- (-2,6.5) -- (4.5,6.5) -- (6.5,0) -- (0,0);

\end{scope}

\draw[thick] (-1.5,0) rectangle (5,6.5);

\node at (-1.5,-0.7) {$V$};

\node at (-4.75,3.5) {Every vector};

\node at (-4.75,2.5) {in $T[v_b]$};

\node at (-4.75,1.5) {is balanced};

\draw[thick,fill] (3.7,5.5) arc (0:360:2mm);

\draw[thick] (3.1,5.5) arc (0:360:0.6);

\draw[thick] (3,6.25) .. controls (6,9) and (13.5,9) .. (16.5,6.25);

\draw[thick,->] (16.5,6.25) .. controls (19,4) and (21,-5) .. (12,-6.8);

\node at (19.25,2) {$f'$};

\node at (2.5,5.5) {$v_b$};

\draw[thick,fill] (3.7,4) arc (0:360:2mm);

\node at (2.25,4) {$T v_b$};

\node at (3.5,2.85) {$\vdots$};

\draw[thick,fill] (3.7,1) arc (0:360:2mm);

\node at (1.75,1) {$T^{a-1}v_b$};

\node at (6.1,6) {$f$};

\draw[thick,->] (4,5.5) -- (9,5.5);

\node at (6.1,5) {$g$};

\draw[thick,<-] (4,4.3) -- (9,5);

% \node at (7.5,4.5) {$f$};

\draw[thick,->] (4,4) -- (9,4);

\draw[thick,<-] (4,2.8) -- (9,3.5);

% \node at (6.0,3) {$f$};

\draw[thick,->] (4,2.5) -- (9,2.5);

% \node at (7.5,1.90) {$g$};

\draw[thick,<-] (4,1.3) -- (9,2);

\draw[thick,->] (4,1) -- (9,1);

\end{scope}

%%%%%%%%%%%%%%%%%%%%%%%%%%%%%%%%%%%%%%%%%%%%%%%%%%%%%

\begin{scope}[shift={(8,0)}]
% \draw[thin, yellow] (0,0) grid (5,5);
% \draw[thin, green] (0,5) grid (5,6);

\draw[thick] (0,0) rectangle (6.5,6.5);

\node at (6.5,-0.7) {$W$};

\node at (12.25,8.25) {Every vector};

\node at (12.25,7.25) {in $T'[fv_b]\setminus\{0\}$};

\node at (12.25,6.25) {is unbalanced};

\draw[thick,fill] (1.7,5.5) arc (0:360:2mm);

\node at (2.75,5.5) {$f v_b$};

\draw[thick,fill] (1.7,4) arc (0:360:2mm);

\node at (3.15,4) {$T' f v_b$};

\node at (1.5,2.5) {$\vdots$};

\draw[thick,fill] (1.7,1) arc (0:360:2mm);

\node at (3.95,1) {$(T')^{a-1} fv_b$};

\end{scope}

%%%%%%%%%%%%%%%%%%%%%%%%%%%%%%%%%%%%%%%%%%%%%%%%%%%%%

\begin{scope}[shift={(0,-8)}]

% \draw[thin,yellow] (0,0) grid (5,6);

\draw[thick] (-1.5,6.5) -- (5,6.5);

\draw[thick] (-1.5,6.5) -- (0.5,0);

\draw[thick] (5,6.5) -- (7,0);

\draw[thick] (0.5,0) -- (7,0);

\node at (-4.5,4.25) {Every vector};

\node at (-4.5,3.25) {in $T[g w_b]\setminus\{0\}$};

\node at (-4.5,2.25) {is unbalanced};

\draw[thick,fill] (2.7,5.5) arc (0:360:2mm);

\node at (0.5,5.5) {$T^{\ell-1}g w_b$};

\node at (3.25,4.35) {\rotatebox[origin=c]{20}{$\vdots$}};

\node at (2.5,2.5) {$Tg w_b$};

\draw[thick,fill] (4.2,2.5) arc (0:360:2mm);

\draw[thick,fill] (4.9,1) arc (0:360:2mm);

\node at (3.45,1) {$g w_b$};

\draw[thick,<-] (3.2,5.5) -- (9,5.5);

\draw[thick,->] (3.9,4.4) -- (9,5);

\draw[thick,<-] (4,4) -- (9.5,3.9);

\draw[thick,->] (4.6,2.85) -- (9.75,3.5);

\draw[thick,<-] (4.6,2.5) -- (10.25,2.5);

\draw[thick,->] (5.25,1.3) -- (10.25,2.1);

\node at (5.75,1.9) {$f$};

\draw[thick,<-] (5.25,0.9) -- (10.75,0.9);

\node at (8.00,0.50) {$g$};

\end{scope}

%%%%%%%%%%%%%%%%%%%%%%%%%%%%%%%%%%%%%%%%%%%%%%%%%%%%%

\begin{scope}[shift={(8,-8)}]
% \draw[thin, yellow] (0,0) grid (5,5);
% \draw[thin, green] (0,5) grid (5,6);

\begin{scope}[shift={(2,0)}]
\draw[thick] (-2,6.5) -- (4.5,6.5);

\draw[thick] (-2,6.5) -- (0,0);

\draw[thick] (4.5,6.5) -- (6.5,0);

\draw[thick] (0,0) -- (6.5,0);
\end{scope}

\node at (12,4.00) {Every vector};

\node at (12,3.00) {in $T'[w_b]$};

\node at (12,2.00) {is balanced};

\draw[thick,fill] (1.8,5.5) arc (0:360:2mm);

\node at (3.85,5.5) {$(T')^{\ell-1} w_b$};

\node at (2.1,4) {\rotatebox[origin=c]{20}{$\vdots$}};

\draw[thick,fill] (3.0,2.5) arc (0:360:2mm);

\node at (4.25,2.5) {$T' w_b$};

\draw[thick,fill] (3.7,1) arc (0:360:2mm);

\node at (4.75,0.8) {$w_b$};

\end{scope}

%%%%%%%%%%%%%%%%%%%%%%%%%%%%%%%%%%%%%%%%%%%%%%%%%%%%%

\end{tikzpicture}
    \caption{Top figure: the map $f$ is an isomorphism. Bottom figure: the map $g$ is an isomorphism. The map $f$ is redefined to $f'$ by changing it on $v_b$ to $f'(v_b)=w_b$. Then using the pair $(f',g)$, every vector in $T'[w_b]$, except $0$, becomes unbalanced.  
    So one has the decompositions 
    $V = T[v_b]\oplus T[gw_b] \oplus \widetilde{V}$ 
    and 
    $W = T'[w_b] \oplus T'[fv_b]\oplus \widetilde{W}$ for some vector spaces $\widetilde{V}$ and $\widetilde{W}$. 
    }
    \label{fig_20001}
\end{figure}
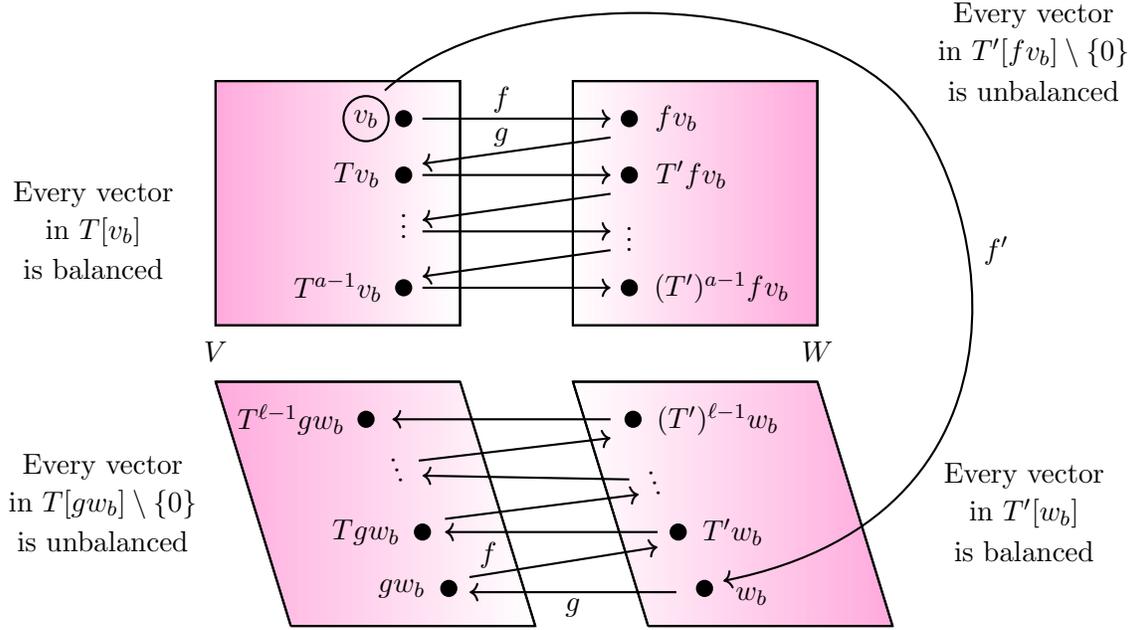

\begin{remark} The difference 
\[
|(v_b,w_b)|- |(v_u,w_u)|  
\]
depends on the isomorphism type of the pair $(f,g)$. We can see this for small dimensions $(m,n)$. Take $m = n =1.$ 
If $(f,g)=(0,0)$, then there is one balanced vector in $V$, so that $|v_b|=1$,  and one balanced vector in $W$, with $|w_b|=1$. The other vectors are unbalanced, 
$|v_u|=q^m-1$, 
$|w_u|=q^n-1$. 
If $(f,g)=(1,0)$, 
every vector in $V$ is balanced, $|v_b|=q^m$ and $|v_u|=0$.  
\end{remark}

\begin{remark}
It would be interesting to investigate if there a relation between the numbers of balanced and unbalanced vectors in nilpotent pairs $(f,g)$, $(f',g')$ and in their direct sum $(f\oplus f',g\oplus g')$ since we see that the dimensions of various kernels above add correctly.
\end{remark}

\begin{example}
\label{ex_small_cases_count_nilp_pairs}
Let $m=0$. Then \eqref{eqn_num_nil_pairs_elegant_expression} reduces to 1.
Let $m=1$. Then \eqref{eqn_num_nil_pairs_elegant_expression} reduces to 
$q^n + q^{n-1} (q^n - 1)
= q^{n-1} (q^n + q -1)$.
For $m=2$, the summation reduces to 
\[ 
q^{2n} 
+ q^{2n-1} (q+1)(q^n-1)
+ q^{2n-2} (q^n-1)(q^n-q) 
= q^{3n-2}(q^n+q^2-1).
\] 
For $m=3$, the summation simplifies as 
\begin{align*}
q^{3n} 
&+ q^{3n-1} \frac{(q^3-1)(q^n-1)}{q-1} 
 + q^{3n-2} \frac{(q^3-1)(q^3-q) (q^n-1)(q^n-q)}{(q^2-1)(q^2-q)}  \\ 
&+  q^{3n-3} (q^n-1)(q^n-q)(q^n-q^2) 
= 
q^{5n-3}(q^n+q^3-1). 
\end{align*} 
\end{example}

\begin{corollary}
\label{cor_prob_nil_pairs}
The probability that $(f,g)$ is nilpotent is 
\begin{equation}
\label{eqn_probability_nilpot_pair}
\dfrac{q^m + q^n - 1}{q^{m+n}} 
= q^{-m} + q^{-n} - q^{-m-n}.
\end{equation}
\end{corollary}

\begin{proof}
This follows immediately by~\eqref{eqn_num_nil_pairs_elegant_expression} and since the cardinality of $\Hom(V,W) \times \Hom(W,V)$ is $q^{2mn}$.
\end{proof}

\begin{example}
\label{ex_probability_small_cases}
Let $m=0$. Then \eqref{eqn_probability_nilpot_pair} reduces to 1.
Let $m=1$. Then \eqref{eqn_probability_nilpot_pair} reduces to 
$\frac{1}{q^n}+ \frac{q^n - 1}{q^{n+1}}=\frac{q^n + q -1}{q^{n+1}}$.
For $m=2$, the summation reduces to 
\[ 
\frac{1}{q^{2n}} +  \frac{(q^2-1)(q^n-1)}{(q-1) q^{2n+1}} + \frac{(q^n-1)(q^n-q)}{q^{2n+2}} 
= \frac{q^n+q^2-1}{q^{n+2}}.
\] 
For $m=3$, the summation simplifies as 
\begin{align*}
\frac{1}{q^{3n}} +  \frac{(q^3-1)(q^n-1)}{(q-1) q^{3n+1}} 
&+ \frac{(q^3-1)(q^3-q) (q^n-1)(q^n-q)}{(q^2-1)(q^2-q) q^{3n+2}}  
+  \frac{(q^n-1)(q^n-q)(q^n-q^2)}{q^{3n+3}} \\
&= 
\frac{q^n+q^3-1}{q^{n+3}}. 
\end{align*} 
\end{example}

\begin{remark}
We can combine Theorem~\ref{thm_gen_nilpotent_two_vs} and Theorem~\ref{thm_dim_nilpotent_pairs} to compute the probability $\mathsf{P}$ that, for a random nilpotent pair $(V,W)$, a randomly chosen $v\in V$ is balanced:
\[
\mathsf{P}(v \ \mathrm{is\ balanced})
    = \frac{|\mathcal{N}(V,W,v_b)|}{|\mathcal{N}(V,W)| |V|} 
    = \frac{q^{2mn}}{q^{2mn - n}(q^m + q^n - 1)}
    = \frac{q^n}{q^m + q^n - 1}.
\]   
\end{remark}

\subsection{Probability in the limit}
\label{subsection_limiting_case}
We continue to work over $\mathbb{F}_q$.
In the following theorem, 
we investigate the limiting quantity of the probability of a nilpotent pair as 
$\dim W\to \infty$ while $\dim V$ remains fixed.

\begin{theorem}
\label{thm_m_fixed_n_large}
For $m$ fixed and in the limit of large $n$, the probability of a nilpotent pair goes to $q^{-m}$.
\end{theorem}

This is the same as the probability computed by Leinster~\cite{Lei21}. The idea is that, as $\dim W$ goes to infinity, probably that a pair of operators $(f,g)$ is nilpotent approaches Leinster's probability that an operator on $V$ is nilpotent.

We will now prove Theorem~\ref{thm_m_fixed_n_large}.

\begin{proof}
Let $m$ be fixed and assume $n\gg m \geq 0$. Then \eqref{eqn_probability_nilpot_pair} simplifies as 
\begin{equation}
\label{eqn_fixed_m}
	\left(\sum_{r=0}^{m}  \qbinom{m}{r}{q} \qbinom{n}{r}{q} q^{-r} 
    \prod_{i=0}^{r-1}(q^r - q^i)  \right)\Big/ q^{mn}.
\end{equation} 
Expanding the above expression, we have 
\begin{align*}
\sum_{r=0}^{m}\frac{\displaystyle{\prod_{i=0}^{r-1}} (q^m-q^i) \displaystyle{\prod_{i=0}^{r-1}}(q^n-q^i) }{\displaystyle{\prod_{i=0}^{r-1}} (q^r-q^i)} \frac{1}{q^{mn+r}}. 
\end{align*}
The leading term for a given $r$ is 
\[
q^{mr+nr-r^2}q^{-mn-r}=q^{r(m+n)-mn-r(r+1)}=q^{-(m-r)(n-r)-r}
\]
which, when $0\le r \le m \ll n$ takes the largest value $q^{-m}$ when $r=m$. A simple manipulation completes the proof of the theorem.  
\end{proof}

\begin{remark}
Alternatively, as a proof of Theorem~\ref{thm_m_fixed_n_large}, we could use the elegant expression~\eqref{eqn_probability_nilpot_pair}  
to see that 
$\lim_{n\rightarrow \infty} (q^{-m}+ q^{-n}-q^{-m-n}) = q^{-m}$.  
\end{remark}

In Proposition~\ref{prop_m_equals_n_large}, we investigate the limiting quantity for \eqref{eqn_probability_nilpot_pair} for $m = n\gg 0$.
\begin{proposition}
\label{prop_m_equals_n_large}
Let $m=n$. Then 
\[
\lim_{n\rightarrow \infty} 
	\left(\sum_{r=0}^{n}  \left(\qbinom{n}{r}{q}\right)^2 
    q^{-r} 
    \prod_{i=0}^{r-1}(q^r - q^i)  \right)\Big/ q^{n^2} = 0.
\]
\end{proposition}

\begin{proof}
Let $m=n$. 
Then \eqref{eqn_fixed_m} simplifies as 
\begin{align*}
\sum_{r=0}^{n}
\frac{\displaystyle{\prod_{i=0}^{r-1}} (q^n - q^i)^2 }{q^{n^2 + r} \displaystyle{\prod_{i=0}^{r-1}} (q^r-q^i)}.
\end{align*}
Let us consider one term in the summation at a time. 
The highest degree term in the numerator of the rational function is $q^{2nr}$. All other terms in the numerator are of lower degree with respect to $q$. 
The maximal degree term in the denominator is  $q^{n^2+r^2+r}$. We want to show $n^2+r^2+r-2nr\geq 0$. This would show that exponent of the leading term in the denominator is greater than the exponent of the leading term in the numerator; thus each term in the summation approaches zero as $n\rightarrow \infty$. 
We have $n^2+r^2+r-2nr = (n-r)^2 + r\geq 0$ since $0\leq r\leq n$. Thus this completes the proof.
\end{proof}

\begin{remark}
Alternatively, we can use 
\eqref{eqn_probability_nilpot_pair} to see that 
\[
\lim_{n\rightarrow \infty} 
    (2 q^{-n} - q^{-2n}) = 0
\]
to obtain Proposition~\ref{prop_m_equals_n_large}.
\end{remark}

\subsection{Nilpotent pairs and balanced vectors over a finite field}
\label{subsection_nilp_pairs_bal_vec_char_q}

Let $k=\mathbb{F}_q$. We resume using the notations from Section~\ref{subsection_nilp_pairs_bal_vec}. For a fixed $m$ and $n$, we want to enumerate the set of all nilpotent pairs $\NVW$ together with length $\ell$ balanced vectors, i.e., balanced vectors $v$ such that $\dim T[v]=\ell$. Let us denote the set of nilpotent pairs together with length $\ell$ balanced vectors as $\mathcal{N}(m,n;\ell)$. We write $|\mathcal{N}(m,n;\ell)|$ to denote the cardinality of the triple $\mathcal{N}(m,n;\ell)$. 
Write $\NVW$ to be the set of all nilpotent pairs between $V$ and $W$, where $\dim V = m$ and $\dim W = n$.

\begin{theorem}
\label{thm_cardinality_nilp_triple}
The cardinality of the set $\mathcal{N}(m,n;\ell)$ of nilpotent pairs with a length $\ell$ balanced vector in $V$ is
\[
\begin{split}
|\mathcal{N}(m,n;\ell)| &= \prod_{i=0}^{\ell-1} (q^m - q^i) \prod_{j=0}^{\ell-1} (q^n - q^j) q^{\ell(m-\ell)}q^{\ell(n-\ell)} \cdot \\ 
&\hspace{2mm} 
\cdot 
q^{2(m-\ell)(n-\ell)-(m-\ell)-(n-\ell)}(q^{m-\ell}+q^{n-\ell}-1).
\end{split}
\]
\end{theorem}

\begin{proof}
There are $q^m-1$ choices for $v$ (since we need to throw out the zero vector), $q^m-q$ choices for $Tv$, and in general, $q^m- q^i$ choices for $T^i v$, where $0\leq i\leq \ell -1$. Similarly, there are $q^n - q^j$ choices for $(T')^j w$, where $0\leq j\leq \ell-1$. Let $V_{m-\ell}=T[v]^{\perp}$. It has dimension $m-\ell$, whose basis vectors can map to any linear combination of the basis vectors of $W$. Similarly, let $W_{n-\ell}=(T'[w])^{\perp}$. It has dimension $n-\ell$, whose basis vectors can map to any linear combination of the basis vectors of $V$. There is however a constraint since $(f,g)$ is a nilpotent pair. So the restricted maps between $V_{m-\ell}$ and $W_{n-\ell}$ must form nilpotent pairs, say $(f',g')$. 
In terms of matrices, we have 
\[
f = 
{\renewcommand{\arraystretch}{1.10}
\begin{blockarray}{ccccccc}
& v & Tv & \ldots & T^{\ell-2}v &T^{\ell-1}v & V_{m-\ell}\\
\begin{block}{c(ccccc|c)}
w & 1 & 0 &\ldots & 0 & 0 & \BAmulticolumn{1}{c}{\multirow{5}{*}{$B$}}\\
T'w & 0 & 1 &\ldots & 0 & 0 & \\
\vdots & \vdots & \vdots & \ddots & \vdots & \vdots &  \\
(T')^{\ell-2}w & 0 & 0 &\ldots & 1 & 0 & \\
(T')^{\ell-1}w & 0 & 0 &\ldots & 0 & 1 & \\
\cline{2-7}
W_{n-\ell} & 0 & 0 &\ldots & 0 & 0 & f' \\ 
\end{block}
\end{blockarray}
}
\]
while 
\[
g = 
{\renewcommand{\arraystretch}{1.10}
\begin{blockarray}{ccccccc}
& w & T'w & \ldots & (T')^{\ell-2}w &(T')^{\ell-1}w  & W_{n-\ell} \\
\begin{block}{c(ccccc|c)}
v & 0 & 0 & \ldots & 0 & 0 & \BAmulticolumn{1}{c}{\multirow{5}{*}{$C$}}\\
Tv &1  & 0 &\ldots  & 0 & 0 & \\
\vdots & \vdots & \vdots & \ddots & \vdots &\vdots  & \\
T^{\ell-2}v & 0 & 0 & \ldots & 0 & 0 & \\
T^{\ell-1}v & 0 & 0 & \ldots & 1 & 0 & \\
\cline{2-7}
V_{m-\ell} & 0& 0 &\ldots &0 &0  & g'\\
\end{block}
\end{blockarray}
}\: ,
\]
where $B$ is an $\ell\times (m-\ell)$-block and $C$ is an $\ell \times (n-\ell)$-block. 
This give us $q^{\ell(m-\ell)}$, $q^{\ell(n-\ell)}$, 
and 
$\mathcal{N}_{m-\ell,n-\ell}$ choices from the $B$ and $C$ blocks and nilpotent pair $(f',g')$, respectively.
We thus obtain a total of: 
\begin{equation} \label{eq_end_of_proof}
|\mathcal{N}(m,n;\ell)| = \prod_{i=0}^{\ell -1} (q^m - q^i) \prod_{j=0}^{\ell -1} (q^n - q^j) q^{\ell(m-\ell)}q^{\ell(n-\ell)} \mathcal{N}_{m-\ell,n-\ell}.
\end{equation} 
Applying Theorem~\ref{thm_dim_nilpotent_pairs} to  
$\mathcal{N}_{m-\ell,n-\ell}$, we are done.
\end{proof}

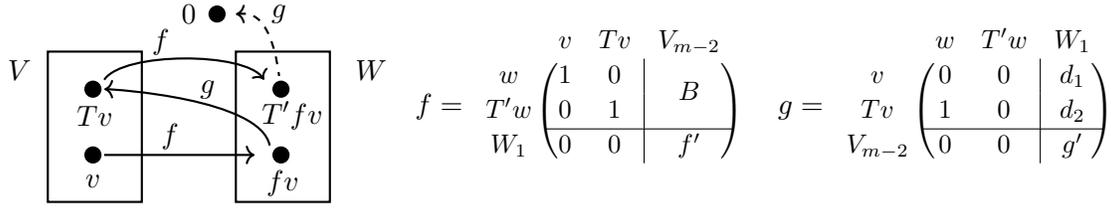
\begin{figure}
    \centering
\begin{tikzpicture}[scale=0.5,decoration={
    markings,
    mark=at position 0.5 with {\arrow{>}}}]

%%%%%%%%%%%%%%%%%%%%%%%%%%%%%%%%%%%%%%%%%%%%%%%%

\begin{scope}[shift={(0,0)}]
% \draw[thin,yellow] (0,0) grid (4,4);
% \draw[thin,green] (4,0) grid (8,4);

\node at (-0.75,3.5) {$V$};

\draw[thick] (0,0) rectangle (2.5,4);

\draw[thick,fill] (1.4,1.25) arc (0:360:2mm);

\node at (1.2,0.55) {$v$};

\node at (3.25,1.75) {$f$};

\draw[thick,->] (1.5,1.25) -- (5.5,1.25);

\draw[thick,fill] (1.4,3) arc (0:360:2mm);

\node at (1.25,2.3) {$Tv$};

\node at (3,4.25) {$f$};

\draw[thick,->] (1.5,3.25) .. controls (2.00,4) and (5.25,4) .. (5.75,3.25);

\node at (4.25,3) {$g$};

\draw[thick,<-] (1.5,3) .. controls (1.75,3) and (5.65,2.75) .. (5.90,1.5);

\draw[thick,dashed,<-] (5,5) .. controls (5.20,5) and (5.90,5.25) .. (6.15,3.35);

\node at (6.15,5) {$g$};

\draw[thick,fill] (4.70,5) arc (0:360:2mm);

\node at (3.75,5) {$0$};

\begin{scope}[shift={(5,0)}]

\node at (3.6,3.5) {$W$};

\draw[thick] (0,0) rectangle (2.5,4);

\draw[thick,fill] (1.4,1.25) arc (0:360:2mm);

\node at (1.25,0.5) {$fv$};

\draw[thick,fill] (1.4,3) arc (0:360:2mm);

\node at (1.5,2.3) {$T'fv$};

\end{scope}

\node at (14.0,2.5) {$f =$ \small 
$
{\renewcommand{\arraystretch}{1.10}
\begin{blockarray}{cccc}
  & v & Tv & V_{m-2}\\ 
\begin{block}{c(cc|c)}
w               & 1 & 0 & 
\BAmulticolumn{1}{c}{\multirow{2}{*}{$B$}} \\ 
T'w             & 0 & 1 &   \\ 
\cline{2-4}
W_1 & 0 & 0 &  f' \\ 
\end{block}
\end{blockarray}
}$};

\node at (23.75,2.5) {$g =$ \small 
${\renewcommand{\arraystretch}{1.10}
\begin{blockarray}{cccc}
  & w & T'w & W_1 \\ 
  \begin{block}{c(cc|c)}
v & 0 & 0 & d_1 \\ 
Tv & 1 & 0 & d_2 \\ 
\cline{2-4}
V_{m-2} & 0 & 0 & g' \\ 
  \end{block}
\end{blockarray}
}
$};

\end{scope}

%%%%%%%%%%%%%%%%%%%%%%%%%%%%%%%%%%%%%%%%%%%%%%%%

\end{tikzpicture}
    \caption{Computations for $\mathcal{N}(m,3;2)$ in Example~\ref{ex_m_3_2}, where $w=fv$,
    $V_{m-2} \simeq (T[v])^{\perp} \simeq \mathbb{F}_q^{m-2}$, $W_1\simeq (T'[fv])^{\perp} \simeq \mathbb{F}_q$, and $(f',g')$ is a nilpotent pair in 
    $\mathcal{N}_{m-2,1}$.}
    \label{fig_00008}
\end{figure}

\begin{example}
\label{ex_m_3_2}
Let $\dim V=m$, $\dim W=3$ and $\ell = 2$. See Figure~\ref{fig_00008}.
Then there are $q^m-1$ choices for $v$, $q^m-q$ choices for $Tv$, $q^3-1$ choices for $fv$, $q^3-q$ choices for $T'(fv)$, $q^{2(m-2)}$ choices for the $B$ block of $f$, $q^2$ choices for $d_1, d_2$, and $\mathcal{N}_{m-2, 3-2}$ choices for the nilpotent pair corresponding to the lower right blocks of $f$ and $g$. 
This results in
\begin{equation}
\begin{split}
|\mathcal{N}(m,3;2)| 
&= (q^m - 1)(q^m - q)(q^3-1)(q^3-q)q^{2(m-2)}q^2 \cdot \mathcal{N}_{m-2, 1}  \\ 
&=
(q^m - 1)(q^m - q)(q^3-1)(q^3-q)q^{2(m-2)}q^2(q^{2m-5}+q^{m-2}-q^{m-3}), 
\end{split}
\end{equation}
which holds for each $m\geq 0$.
\end{example}

\begin{figure}
    \centering
\begin{tikzpicture}[scale=0.5,decoration={
    markings,
    mark=at position 0.5 with {\arrow{>}}}]

%%%%%%%%%%%%%%%%%%%%%%%%%%%%%%%%%%%%%%%%%%%%%%%%%

\begin{scope}[shift={(0,0)}]
% \draw[thin,yellow] (0,0) grid (4,4);
% \draw[thin,green] (4,0) grid (8,4);

\node at (-0.75,3.5) {$V$};

\draw[thick] (0,0) rectangle (2.5,3.5);

\draw[thick,fill] (1.4,1.25) arc (0:360:2mm);

\node at (1.25,0.5) {$v$};

\node at (4,2.75) {$f$};

\draw[thick,->] (1.5,1.5) .. controls (2,2.25) and (5.25,2.25) .. (5.75,1.5);

\draw[thick,dashed,<-] (4.5,4.5) .. controls (4.75,4.5) and (5.75,4.75) .. (6,2);

\node at (5.75,4.5) {$g$};

\draw[thick,fill] (4.2,4.5) arc (0:360:2mm);

\node at (3.25,4.5) {$0$};

\node at (-1.5,2.5) {$\dim m$};

\node at (9,2.5) {$\dim 2$};

\begin{scope}[shift={(5,0)}]

\node at (3.75,3.5) {$W$};

\draw[thick] (0,0) rectangle (2.5,3.5);

\draw[thick,fill] (1.4,1.25) arc (0:360:2mm);

\node at (1.25,0.5) {$w$};

\end{scope}

\node at (16,1.5) {$f = \kbordermatrix{
  & v & V_{m-1} \\ 
w & 1 & B \\ 
W_{1} & 0 & f' 
}$};

\node at (24,1.5) {$g = \kbordermatrix{
  & w & W_{1} \\ 
v & 0 & c \\ 
V_{m-1} & 0 & g' 
}
$};

\end{scope}
%%%%%%%%%%%%%%%%%%%%%%%%%%%%%%%%%%%%%%%%%%%%%%%%

\begin{scope}[shift={(0,-6.5)}]
% \draw[thin,yellow] (0,0) grid (4,4);
% \draw[thin,green] (4,0) grid (8,4);

\node at (-0.75,3.5) {$V$};

\draw[thick] (0,0) rectangle (2.5,4);

\draw[thick,fill] (1.4,1.25) arc (0:360:2mm);

\node at (1.25,0.5) {$v$};

\node at (3.25,1.75) {$f$};

\draw[thick,->] (1.5,1.25) -- (5.5,1.25);

\draw[thick,fill] (1.4,3) arc (0:360:2mm);

\node at (3,4.25) {$f$};

\draw[thick,->] (1.5,3.25) .. controls (2.00,4) and (5.25,4) .. (5.75,3.25);

\node at (4.25,3) {$g$};

\draw[thick,<-] (1.5,3) .. controls (1.75,3) and (5.65,2.75) .. (5.90,1.5);

\draw[thick,dashed,<-] (5,5) .. controls (5.20,5) and (5.90,5.25) .. (6.15,3.35);

\node at (6.15,5) {$g$};

\draw[thick,fill] (4.70,5) arc (0:360:2mm);

\node at (3.75,5) {$0$};

\begin{scope}[shift={(5,0)}]

\node at (3.75,3.5) {$W$};

\draw[thick] (0,0) rectangle (2.5,4);

\draw[thick,fill] (1.4,1.25) arc (0:360:2mm);

\node at (1.25,0.5) {$w$};

\draw[thick,fill] (1.4,3) arc (0:360:2mm);

\end{scope}

\node at (15.25,1.5) {$f =$ \small 
$\begin{blockarray}{cccc}
                & v & Tv & V_{m-2}\\ 
    \begin{block}{c(cc|c)}
w               & 1 & 0  & \BAmulticolumn{1}{c}{\multirow{2}{*}{$D$}} \\ 
T'w             & 0 & 1  &  \\ 
    \end{block}
\end{blockarray}$
};

\node at (24,1.5) {$g =$ \small 
${\renewcommand{\arraystretch}{1.10} 
\begin{blockarray}{ccc}
  & w & T'w  \\ 
\begin{block}{c(cc)}
v & 0 & 0  \\ 
Tv & 1 & 0 \\
\cline{2-3}
V_{m-2} & 0 & 0 \\
\end{block}
\end{blockarray}
}
$};

\end{scope}

%%%%%%%%%%%%%%%%%%%%%%%%%%%%%%%%%%%%%%%%%%%%%%%%

\end{tikzpicture}
    \caption{Top row: computations for $\mathcal{N}(m,2;1)$ for Example~\ref{ex_f_g_l_m2}, where $w=fv$,
    $V_{m-1} \simeq \Span\{v\}^{\perp} \simeq \mathbb{F}_q^{m-1}$, $W_{1} \simeq \Span\{ w\}^{\perp} \simeq \mathbb{F}_q$,
    $B$ is a $1\times (m-1)$ block, $c$ is a single matrix coordinate, and $(f',g')$ is a nilpotent pair in $\mathcal{N}_{m-1, 1}$. 
    Second row: computations for $\mathcal{N}(m,2;2)$ for Example~\ref{ex_f_g_l_m2}, where $V_{m-2} \simeq (T[v])^{\perp} \simeq \mathbb{F}_q^{m-2}$, and $D$ is a $2\times (m-2)$ block.}
    \label{fig_00007}
\end{figure}
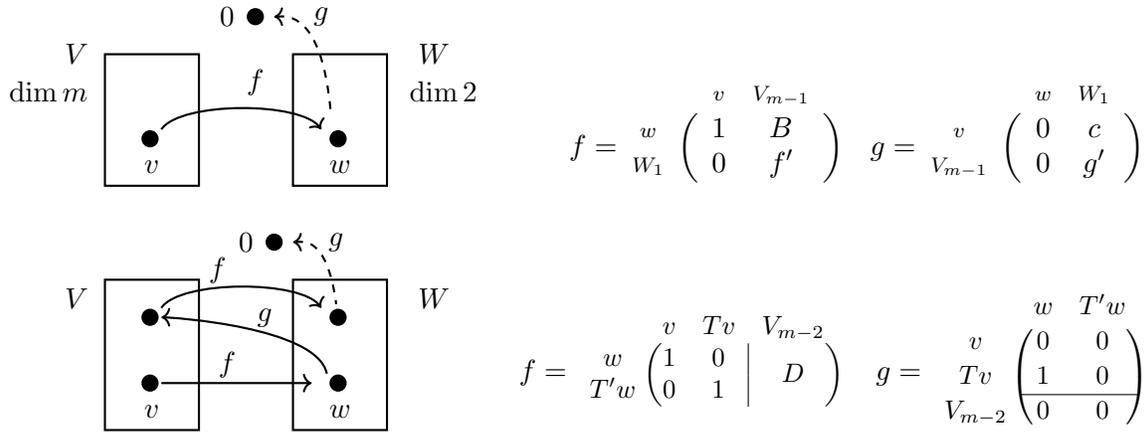
 
\begin{example}
\label{ex_f_g_l_m2}
Let $n=2$. 
Then $|\mathcal{N}(m,2;0)| = \mathcal{N}_{m,2} = q^{4m-2}+q^{3m}-q^{3m-2}$ by Theorem~\ref{thm_dim_nilpotent_pairs}. 
Next, 
$|\mathcal{N}(m,2;1)| 
= (q^m-1)(q^2-1)q^{m-1}q \cdot \mathcal{N}_{m-1, 1}
= (q^m-1)(q^2-1)q^{m}(q^{2(m-1)-1}+q^{m-1}-q^{m-2})$ by the first row in Figure~\ref{fig_00007} and Theorem~\ref{thm_dim_nilpotent_pairs}, 
where 
$\mathcal{N}_{m-1, 1}$ is the cardinality of the set of nilpotent pairs for the vector spaces $V_{m-1}$ and $W_{1}$. So there are $q^m-1$ choices for $v$, $q^2-1$ choices for $w$, $q^{m-1}$ choices for $B$ since $B$ is a $1\times (m-1)$ block, $q$ choices for $c$, and $|(m-1,1)|$ choices for the nilpotent pair $(f',g')$.

Next, we have 
$|\mathcal{N}(m,2;2)| = (q^m-1)(q^m-q)(q^2-1)(q^2-q)q^{2(m-2)}$ by the second row in Figure~\ref{fig_00007} since there are $q^m-1$ choices for $v$, $q^2-1$ choices for $w$, $q^m-q$ choices for $Tv$, $q^2-q$ choices for $T'w$, and $q^{2(m-2)}$ choices for the $D$ block.

This gives us $|\mathcal{N}(m,2;0)| +|\mathcal{N}(m,2;1)| +|\mathcal{N}(m,2;2)| = q^{4m}=|\Hom(V,W)|^2$, as expected.
\end{example}

% \MeeSeong{We learned that $\Delta$-matroids are a generalization of matroids and matroids include planar graphs. Perhaps discuss the correspondence between nilpotent matrices and $\Delta$-matroids? Make this correspondence precise. Or add this as a remark?}

\section{Appendix}
\label{section_appendix}
In this section, we provide two \texttt{Mathematica} codes for Proposition~\ref{prop_nilp_bool_semiring}, which produce the set of all pairs of $n\times n$ matrices that are nilpotent over the Boolean semiring. See Figure~\ref{fig_code_boolean}.

\begin{figure}[h]
\centering
\begin{verbatim}
n = 3;
powersOfMatrices = Tuples[{0, 1}, n^2] // Map[Partition[#, n] &];
pairs = Select[
   Table[{mat, Unitize[MatrixPower[mat, n]]}, {mat, powersOfMatrices}], 
   Last[#] === ConstantArray[0, {n, n}] &];
MatrixForm[pairs]
Length[pairs]
\end{verbatim}
\caption{This \texttt{Mathematica} code produces the set of all pairs of the form 
$(x,x^n)  \in \M_n(\mathbb{B})^2$ such that $(x,x^n) = (x,0)$, i.e., 
$x$ is a Boolean nilpotent operator.}
\label{fig_code_boolean}
\end{figure}

If we want to produce only the first coordinate in each pair of matrices in Figure~\ref{fig_code_boolean}, use the code in Figure~\ref{fig_nilpotent_mat_only}.

\begin{figure}[ht]
    \centering
\begin{verbatim}
n = 3;
powersOfMatrices = Tuples[{0, 1}, n^2] // Map[Partition[#, n] &];
nilpotentMatrices = 
  Select[powersOfMatrices, 
   Unitize[MatrixPower[#, n]] === ConstantArray[0, {n, n}] &];
nilpotentMatrices
Length[nilpotentMatrices]
\end{verbatim}
    \caption{This \texttt{Mathematica} code returns only  $n\times n$ nilpotent matrices over $\mathbb{B}$.}
    \label{fig_nilpotent_mat_only}
\end{figure}

%%%%%%%%%%%%%%%%%%%%%
%%
%%   REFERENCES 
%%
%%%%%%%%%%%%%%%%%%%%

\FloatBarrier

\bibliographystyle{amsalpha} 
\bibliography{nilpotent_finite}

\end{document}